\documentclass[12pt]{article}

\usepackage{amssymb,amsmath,amsthm,ucs}

\newcommand{\ovl}{\overline}
\newcommand{\n}{\noindent}
\newcommand{\vp}{\varepsilon}

\newcommand{\ovltimes}{\ovl\otimes}
\newcommand{\cp}{M\rtimes_\alpha G}
\newcommand{\rcp}{M\rtimes_{\alpha,r} G}
\newcommand{\cpcut}{p_1Mp_1\rtimes_\alpha G}
\newcommand{\rcpcut}{p_1Mp_1\rtimes_{\alpha,r} G}
\newcommand{\rcpl}{M\rtimes_{\alpha,r} G_\lambda}

\newcommand{\sett}[1]{\left\{#1\right\}}

\newcommand{\eps}{\varepsilon}

\newcommand{\U}{\mathcal{U}}

\newcommand{\norm}[1]{\left\Vert#1\right\Vert}

\numberwithin{equation}{section}

\theoremstyle{plain}
\newtheorem{lem}{Lemma}[section]

\newtheorem{thm}[lem]{Theorem}
\newtheorem{cor}[lem]{Corollary}

\theoremstyle{definition}

\theoremstyle{remark}
\newtheorem{rem}[lem]{Remark}

\begin{document}

\null\vspace{.5in}
\setcounter{page}{1}
\begin{LARGE}
\begin{center}
BIMODULES IN CROSSED PRODUCTS AND\\
REGULAR INCLUSIONS OF FINITE FACTORS
\end{center}
\end{LARGE}
\bigskip

\begin{center}
\begin{tabular}{cc}
{\large Jan Cameron}$^{(*)}$&{ \large Roger R. Smith}$^{(**)}$\\
&\\
Department of Mathematics & Department of Mathematics\\
Vassar College& Texas A\&M University\\
Poughkeepsie, NY 12604 & College Station, TX 77843\\
&\\
jacameron@vassar.edu & rsmith@math.tamu.edu

\end{tabular}
\end{center}

\vfill

\begin{abstract}
In this paper, we study bimodules over a von Neumann algebra $M$ in two related contexts. The first is an inclusion $M \subseteq M \rtimes_\alpha G$, where $G$ is a discrete group
acting on a factor $M$ by outer automorphisms. The second is a regular inclusion $M \subseteq N$ of finite factors.  In the case of crossed products, we characterize the $M$-bimodules $X$ that lie between $M$ and $M \rtimes_\alpha G$ and are closed in the Bures topology, in terms of the subsets of $G$.  We show that this
characterization also holds for $w^*$-closed bimodules when $G$ has the approximation property ($AP$), a class of groups that includes all amenable and weakly amenable ones.  As an application, we prove a version of
Mercer's extension theorem for certain $w^*$-continuous isometric maps on $X$. We establish a similar  theorem for bimodules arising from regular inclusions of finite
factors, which generalizes the crossed product situation when $G$ acts on a finite factor.  In the final section we apply these ideas to provide new examples of singly generated finite factors.

\end{abstract}

\vfill

\noindent Key Words:\qquad von Neumann algebra, crossed product, bimodule

\vfill

\noindent AMS Classification: 46L10, 46L06

\vfill

\noindent ($\ast$) JC was partially supported by an AMS-Simons research travel grant.

\noindent ($\ast \ast$) RS was partially supported by NSF grant DMS-1101403. Corresponding author.
\newpage

\section{Introduction}\label{sec1}
The starting point for this paper is a theorem due to H. Choda \cite{Cho} which describes the von Neumann algebras which lie between a factor $M$ and its crossed product
$M\rtimes_\alpha G$ by a discrete group $G$ acting on $M$ by outer automorphisms. Choda proved that an intermediate von Neumann algebra $N$ for which there exists a
normal conditional expectation $E:\cp \to N$ must have the form $M\rtimes_\alpha H$ for a subgroup $H$ of $G$. When $M$ is type II$_1$, the existence of $E$ is immediate
and so the intermediate von Neumann algebras are precisely of the form
$M\rtimes_\alpha H$ for subgroups $H$ of $G$. The result for other types of factors was improved in \cite{ILP}, where it was shown that such conditional expectations
always exist, at least when $M$ has separable predual. It is natural to ask whether the intermediate $w^*$-closed $M$-bimodules can be characterized in a similar way.
Each unital subset $S$ of $G$ gives rise to a $w^*$-closed $M$-bimodule as
$X_S=\ovl{{\text{span}}}^{w^*}\sett{mg:m\in M,\ g\in S}$, so the question is whether all $w^*$-closed $M$-bimodules occur in this way. The first part of the paper is
devoted to this problem, and we are able to provide a complete answer when we work with a different topology introduced by Bures in \cite{Bur}. Our techniques are
largely based on Fourier series in the crossed product and, as noted by Mercer \cite{Mer0}, the Bures topology is the correct one for understanding the convergence of
such series. We establish a correspondence between the Bures closed $M$-bimodules and the subsets of $G$, and recapture the results on intermediate von Neumann algebras
above without separability restrictions. The Bures closed $M$-bimodules are all $w^*$-closed, but the reverse statement is open. However, for groups with a finite
Haagerup constant (a class of groups known as the weakly amenable groups, which contains all amenable groups) the two classes of bimodules coincide.

Bimodules over subalgebras of von Neumann algebras have been studied previously in various contexts. For example, the case of an inclusion $A \subseteq M$ for a Cartan subalgebra $A$ was examined
in \cite{MSS}, but unfortunately the arguments for the Spectral Theorem for Bimodules contained a gap, as pointed out in \cite{Aoi}. A proof for the case when the containing von Neumann algebra $M$ is
amenable was given in \cite{Ful}.  A full rehabilitation of the Spectral Theorem for Bimodules was then attempted in \cite{Mer}, but this regrettably also contained an error. A slightly different formulation of the result was
found recently in \cite{CPZ} which characterized the Bures closed $A$-bimodules in $M$ rather than the $w^*$-closed ones, but the status of the original formulation of the Spectral Theorem for Bimodules  remains uncertain. There are certainly $w^*$-closed subspaces that are
not Bures closed, suggesting that the same phenomenon might occur for bimodules, but no such examples are known.    We also note that bimodules have been studied in the related context of tensor products of von Neumann algebras by Kraus \cite{Kra}. If $N$ is a factor satisfying a technical condition called the {\emph{weak}}$^*$ {\emph{operator approximation property}} and $R$ is a von Neumann algebra, then he was able to characterize the $\sigma$-weakly closed $N$-bimodules of $N\overline{\otimes}R$ as those subspaces of the form $N\overline{\otimes}T$ where $T$ is a $\sigma$-weakly closed subspace of $R$.

If $X$ is a $w^*$-closed unital subspace of a von Neumann algebra which it generates, then we may consider the inclusion $X \subseteq W^*(X)$. An old question, dating
back at least to early work of Arveson \cite{Ar1,Ar2}, asks whether $w^*$-continuous unital isometries on $X$ extend to $\ast$-automorphisms of $W^*(X)$. This has been studied in
\cite{DavP,CPZ}, the latter paper concentrating on the situation $A\subseteq X\subseteq M$ where $A$ is Cartan in $M$ and $X$ is a $w^*$-closed $A$-bimodule that
generates $M$. Mercer \cite{Mer} claimed that isometric unital maps on $X$ that respected the bimodule structure extended to $\ast$-automorphisms of $M$, but the proof
was flawed. If such a result were true then the original map would have to be $w^*$-continuous, so the correct formulation of this theorem is to assume $w^*$-continuity,
and the extension result was proved in \cite{CPZ} under this hypothesis. In this paper we prove a version of Mercer's theorem in the context of bimodules over a von
Neumann factor.

The structure of the paper is as follows. Section \ref{sec2} recalls the properties of crossed products by discrete groups, and also some facts about weakly amenable
groups. Section \ref{sec3} is devoted to a brief discussion of the Bures topology and its relationship to the $w^*$-topology. Section \ref{sec4} deals with 
$M$-bimodules $X$ for an inclusion $M \subseteq \cp$, and characterizes those that are closed in the Bures topology as being in bijective correspondence with the
unital subsets $S$ of $G$ via the map
\[S\mapsto \ovl{{\text{span}}}^B\sett{mg:m\in M,\ g\in S}.\]
In the special case of a weakly amenable group, the map
\[S\mapsto \ovl{{\text{span}}}^{w^*}\sett{mg:m\in M,\ g \in S}\]
also identifies subsets of $G$ with the $w^*$-closed $M$-bimodules.  Our techniques for proving these and subsequent results depend crucially on a theorem of Christensen and Sinclair \cite{CS} which, in the form used here, constructs a projection from the space of completely bounded maps on a von Neumann algebra $M$ to the space of right $M$-module maps. Not only the existence of this projection but also the method of constructing it will be important for this paper (see Theorem \ref{thm4.1}).

Section \ref{sec5} examines $M$-bimodules $Y$ inside the reduced C$^*$-algebra crossed product $\rcp$ and which generate this C$^*$-algebra. We show that $\rcp$ is the
C$^*$-envelope of such bimodules $Y$, after showing that a nontrivial norm closed ideal in $\rcp$ must have nontrivial intersection with $M$. These are the preliminary
results needed for a version of Mercer's theorem for $M$-bimodules $X$ satisfying $M\subseteq X\subseteq \cp$, proved in Theorem \ref{thm6.3}.

Inclusions $M\subseteq \cp$ are particular examples of regular inclusions $M\subseteq N$ of factors. Section \ref{sec7} takes our earlier results on crossed products and
extends them to regular inclusions of finite factors. Formally, the results are the same but require somewhat different methods in this more general context. The last
section applies our results to a different problem, the question of single generation of von Neumann algebras. Here we establish some new classes of finite factors for
which a positive answer can be given, based on an invariant introduced by Shen \cite{Shen}.

\section{Crossed products}\label{sec2}

A $\ast$-automorphism $\theta$ of a von Neumann algebra $M$ is said to be {\emph{properly outer}} if there does not exist a nonzero projection $p \in M$ so that $\theta$
restricts to an inner $\ast$-automorphism of $pMp$. In the case of a factor, this is equivalent to $\theta $ being {\emph{outer}}, meaning that it is not inner. It is well known
that proper outerness is equivalent to {\emph{freeness}} \cite[$\S$ 17]{Str} which is defined by the condition that no nonzero element $t\in M$ can satisfy the equation $tx=\theta(x)t$ for
$x\in M$. We note that Lemma \ref{lem5.0} below gives a useful variant on freeness that also characterizes proper outerness of $\ast$-automorphisms. When we say that a
group $G$ acts by (properly) outer automorphisms on a von Neumann algebra $M$, we of course intend this to apply to each automorphism $\alpha_g$ for $g\ne e$ since
$\alpha_e$ is always the identity which is trivially inner. The following discussion of crossed products and Lemma \ref{lem2.1} do not require properly outer actions,
but such an assumption is necessary from Section \ref{sec4} onwards.

Let $G$ be a discrete group that acts on a von Neumann algebra $M\subseteq B(H)$ by automorphisms $\alpha_g$, $g\in G$. We may define a faithful normal representation $\pi:M\to
B(H\otimes \ell^2(G))$ by
\begin{equation}\label{eq2.1}
\pi(x)(\xi\otimes \delta_h)=\alpha_{h^{-1}}(x)\xi\otimes \delta_h,\ \ \ \ x\in M,\ \xi\in H,\ h\in G,
\end{equation}
and a unitary representation of the group $G$ by
\begin{equation}\label{eq2.2}
\lambda_g(\xi\otimes \delta_h)=\xi\otimes \delta_{gh},\ \ \ \ \xi\in H,\ g,h\in G,
\end{equation}
which is $I\otimes \ell_g$ for the left regular representation $g\mapsto \ell_g$ of $G$ on $\ell^2(G)$.
These operators are chosen so that
\begin{equation}\label{eq2.3}
\lambda_g\pi(x)\lambda_{g^{-1}}=\pi(\alpha_g(x)),\ \ \ \ x\in M,\ g\in G,
\end{equation}
and so there is no loss of generality in assuming that $M$, in its original representation, has its automorphisms $\alpha_g$ spatially implemented by a unitary representation $g\mapsto
a_g$, $g\in G$. This is standard theory, and we follow the presentation and notation of \cite{vD}. The crossed product $\cp$ is then the von Neumann algebra in $B(H\otimes \ell^2(G))$
generated by the operators $\pi(x)$ and $\lambda_g$ for $x\in M$ and $g\in G$.

If we define a unitary operator $W\in B(H\otimes \ell^2(G))$ by
\begin{equation}\label{eq2.4}
W(\xi\otimes \delta_h)=a_h\xi\otimes \delta_h,\ \ \ \ \xi\in H, \ h\in G,
\end{equation}
then
\begin{equation}\label{eq2.5}
W\pi(x)W^*=x\otimes I,\ \ \ \ x\in M,
\end{equation}
and
\begin{equation}\label{eq2.6}
W\lambda_gW^*=a_g\otimes \ell_g,\ \ \ \ g\in G,
\end{equation}
(see \cite[Prop. 2.12]{vD}). Note that \eqref{eq2.5} and \eqref{eq2.6} imply that $W(\cp)W^*\subseteq B(H)\,\ovl{\otimes}L(G)$, where $L(G)$ is the group von Neumann algebra. The predual
of $L(G)$ is the Fourier algebra $A(G)$ consisting of the functions $g\mapsto \langle \ell_g\xi,\eta\rangle$ where $\xi,\eta\in \ell^2(G)$. A {\emph{completely bounded multiplier}} is a function
$f$ on $G$ such that $f\cdot A(G)\subseteq A(G)$ and the associated multiplication operator $M_f$ is completely bounded. The space of completely bounded multipliers is denoted by $M_0(A(G))$ and
is a Banach space in the {\it{cb}}-norm. In \cite{HK} it is shown that
$M_0(A(G))$ is a dual space with predual denoted by $Q(G)$, and $G$ is said
to have the {\emph{approximation property}} ($AP$) if the constant function
1 is in the $w^*$-closure of the space of finitely supported functions on
$G$. This is a large class of groups that contains all amenable discrete
groups and also the weakly amenable groups of \cite{CH}.

The operators in part (vi) of the following lemma are only given a concrete description on the generators of $\cp$. The extension to general elements will be established in equation \eqref{eq4.21}. 

In an earlier version of the paper, this lemma was formulated for weakly
amenable discrete groups. We thank Jon Kraus for pointing out that
essentially the same proof works for groups with the $AP$.
\begin{lem}\label{lem2.1}
Let $G$ be a discrete  group with the $AP$ acting by automorphisms $\alpha_g$, $g\in G$, on a von Neumann algebra $M\subseteq B(H)$. Then there exist  a net
$(f_\gamma)_{\gamma\in\Gamma}$ of finitely supported functions on $G$ and a net $(T_\gamma:\cp\to\cp)_{\gamma\in\Gamma}$ of normal maps with the following properties:
\begin{itemize}
\item[\rm (i)]
For each $\gamma\in \Gamma$, $M_{f_\gamma}$ is completely bounded.
\item[\rm (ii)]
For each $g\in G$, $\underset{\gamma}{\lim}\,  f_\gamma(g)=1$.
\item[\rm (iii)]
For each function $h\in A(G)$, $\underset{\gamma}{\lim}\,  \|M_{f_\gamma}h-h\|=0$.
\item[\rm (iv)]
For each $\gamma\in\Gamma$, $T_\gamma$ is completely bounded.
\item[\rm (v)]
For each $y\in \cp$, $\underset{\gamma}{\lim}\,  T_\gamma(y)=y$ in the $w^*$-topology.
\item[\rm (vi)]
For each $x\in M$ and $g\in G$, $T_\gamma(\pi(x)\lambda_g)=f_\gamma(g)\pi(x)\lambda_g$.
\end{itemize}
\end{lem}
\begin{proof}
The existence of such a net of functions satisfying (i)-(iii) is \cite[Theorem 1.9]{HK}. The predual of $B(H)\ovl{\otimes}L(G)$ is the operator space projective tensor
product $B(H)_*\,\hat{\otimes}_{op}\,A(G)$, \cite{ER}. Thus the operators $I\otimes M_{f_\gamma}$ on $B(H)_*\,\hat{\otimes}_{op}\,A(G)$ are  completely bounded  and converge
to $I\otimes I$ in the point norm topology. Let $S_\gamma$ denote the normal map
$(I\otimes M_{f_\gamma})^*:B(H)\ovl{\otimes}L(G)\to B(H)\ovl{\otimes}L(G)$. Then $\underset{\gamma}{\lim}\,  S_\gamma =I$ in the point $w^*$-topology and each $S_\gamma$ is completely bounded.

Using the operator $W$ of \eqref{eq2.4}, we define $T_\gamma:\cp\to W^*(B(H)\ovl{\otimes}L(G))W$ by
\begin{equation}\label{eq2.7}
T_\gamma(y)=W^*S_\gamma(WyW^*)W,\ \ \ \ \gamma\in \Gamma,\ y\in \cp.
\end{equation}
Clearly each $T_\gamma$ is a normal map. From the duality between $A(G)$ and $L(G)$, we see that $M_{f_\gamma}(\ell_g)=f_\gamma(g)\ell_g$, and so
\begin{equation}\label{eq2.8}
S_\gamma(x\otimes \ell_g)=x\otimes f_\gamma(g)\ell_g,\ \ \ \ x\in B(H),\ g\in G.
\end{equation}
Thus
\begin{align}
T_\gamma(\pi(x)\lambda_g)&=W^*S_\gamma(xa_g\otimes \ell_g)W=W^*(xa_g\otimes f_\gamma(g)\ell_g)W\notag\\
&=f_\gamma(g)\pi(x)\lambda_g,\ \ \ \ \gamma\in \Gamma,\ x\in M,\ g\in G,\label{eq2.9}
\end{align}
using \eqref{eq2.5}-\eqref{eq2.8}. Then normality shows that $T_\gamma$  maps $\cp$ to itself and establishes (vi). The remaining properties (iv) and (v) follow from the corresponding
properties of $S_\gamma$.
\end{proof}
Henceforth we will ease notation and write generating elements of $\cp$ as $xg$, $x\in M$, $g\in G$, subject to the relation $gx=\alpha_g(x)g$.

\section{The Bures topology}\label{sec3}
In this section we investigate a topology, introduced by Bures \cite{Bur}, which was shown by Mercer \cite{Mer} to be the appropriate one for questions of convergence of Fourier series in crossed products. The setting is an inclusion $M\subseteq N$ of von Neumann algebras for which there exists a faithful, normal conditional expectation $E:N\to M$. Such a conditional expectation always exists when $N$ is the crossed product $\cp$ of $M$ by a discrete group $G$ and this is the primary situation where we will use this topology. The {\emph{Bures topology}}, or $B$-{\emph{topology}} as we will call it, is defined by the seminorms $x\mapsto \phi(E(x^*x))^{1/2}$, $x\in N$, where $\phi$ ranges over the normal states of $M$. Thus a set $U\subseteq N$ is $B$-open if, and only if, given $x_0\in U$, there exist normal states $\phi_1,\ldots,\phi_n\in M_*$ and $\vp >0$ so that 
\begin{equation}\label{eq3.1}
\cap_{i=1}^n \sett{y:\phi_i(E((y-x_0)^*(y-x_0)))<\vp^2}\subseteq U.
\end{equation}
The objective of the following two lemmas is to establish the relationship of the $B$-topology to the $w^*$-topology.
\begin{lem}\label{lem3.1}
Let $M\subseteq N$ be von Neumann algebras with a faithful, normal conditional expectation $E:N\to M$. Let $K$ be a convex subset of $N$. If $K$ is $B$-closed, then it is $w^*$-closed.
\end{lem}
\begin{proof}
Suppose not. Then there exists $x_0\in \ovl{K}^{w^*}\setminus K$. By hypothesis, the complement $K^c$ is $B$-open and so there exist normal states $\phi_1,\ldots,\phi_n\in M_*$ and $\vp>0$ so that 
\begin{equation}\label{eq3.2}
\cap_{i=1}^n\sett{y:\phi_i(E((y-x_0)^*(y-x_0)))<\vp^2}\subseteq K^c.
\end{equation}
Define a normal state $\phi\in M_*$ as the average of the $\phi_i$'s, and let $\delta=\vp/\sqrt{n}$. Then \eqref{eq3.2} implies that
\begin{equation}\label{eq3.3}
\sett{y:\phi(E((y-x_0)^*(y-x_0)))<\delta^2}\subseteq K^c.
\end{equation}
On $N$, define a semi-inner product by
\begin{equation}\label{eq3.4}
\langle x,y\rangle=\phi(E(y^*x)),\ \ \ \ x,y\in N,
\end{equation}
and let $X\subseteq N$ be the subspace of null vectors. Then the quotient space $N/X$ has an induced inner product and we write $H$ for the Hilbert space completion. Let $\hat K$ be the image of $K$ in $H$ and note that \eqref{eq3.3} gives
\begin{equation}\label{eq3.5}
d(x_0+X,\hat K)\geq \delta.
\end{equation}
By Hahn-Banach separation there exist real constants $\alpha$ and $\beta$ with $\beta >0$ and a vector $\xi\in H$ so that 
\begin{equation}\label{eq3.6}
{\mathrm{Re}}\,\langle x_0+X,\xi\rangle \geq \alpha+\beta>\alpha \geq {\mathrm{Re}}\,\langle k+X,\xi\rangle,\ \ \ \ k\in K.
\end{equation}
Define a linear functional $\psi: N\to \mathbb{C}$ by
\begin{equation}\label{eq3.7}
\psi(x)=\langle x+X,\xi\rangle,\ \ \ \ x\in N.
\end{equation}
Then \eqref{eq3.6} becomes
\begin{equation}\label{eq3.8} 
{\mathrm{Re}}\,\psi(x_0)\geq \alpha+\beta >\alpha\geq {\mathrm{Re}}\,\psi (k),\ \ \ \ k\in K,
\end{equation}
and so it suffices to show that $\psi\in N_*$ in order to reach a contradiction.
Choose a sequence $y_n\in N$, $n\geq 1$, so that $y_n+X\to \xi$ in $H$ and $\|y_n+X\|_H\leq \|\xi\|_H$. Define $\psi_n\in N_*$ by $\psi_n(x)=\phi(E(y^*_nx))$, $x\in N$. Then, for $x\in
N$,
\begin{align}
|\psi_n(x)-\psi(x)|^2&=|\langle x+X,(y_n+X)-\xi\rangle|^2\notag\\
&\leq \|x\|^2_H\|(y_n+X)-\xi\|^2_H\notag\\
&\leq \|x\|^2\|(y_n+X)-\xi\|^2_H,\label{eq3.9}
\end{align}
showing that $\|\psi_n-\psi\|\to 0$. Thus $\psi\in N_*$, completing the proof.
\end{proof}

\begin{lem}\label{lem3.2}
Let $M\subseteq N$ be von Neumann algebras with a faithful, normal conditional expectation $E:N\to M$. Let $Y\subseteq N$ be a norm bounded subset. If $Y$ is $w^*$-closed, then it is $B$-closed.
\end{lem}
\begin{proof}

Suppose that $x_0 \in N$ lies in the $B$-closure of $Y$.  Then there is a net $(x_{\lambda})_{\lambda \in \Lambda}$ in $Y$ converging to $x_0$ in the $B$-topology, and $w^*$-compactness allows us to drop to a subnet which also converges in the $w^*$-topology to some $y_0 \in Y.$   We will show that $x_0 = y_0.$  

Given an arbitrary normal state $\phi \in M_*,$ and $z \in N$, define $\phi_z \in N_*$ by 

\begin{equation}\label{eq3.10}
\phi_z(x)=\phi(E(zx)),\ \ \ \ x\in N,
\end{equation}
and denote by $L$ the span of these functionals in $N_*$ as $\phi$ and $z$ vary.  We claim that $L$ is norm dense in $N_*$.  Indeed, if it were not, then there would exist a nonzero $x_1 \in N$ such that

\begin{equation}\label{eq3.11}
\phi(E(zx_1))=0
\end{equation}
 for every normal state $\phi \in M_*$ and $z \in N.$  Taking $z = x_1^*$ and letting $\phi$ vary over $M_*,$ the faithfulness of $E$ would then imply that $x_1 = 0,$ a contradiction.    Thus, $L$ is dense in $N_*$.  
 
If $\phi$ is a normal state in $M_*$ and $z\in N$, then applying the Cauchy-Schwarz inequality to the state $\phi\circ E$ gives
\begin{align}
|\phi(E(z(x_\lambda -x_0)))|&\leq \phi(E(zz^*))^{1/2}\phi(E((x_\lambda -x_0)^*(x_\lambda -x_0)))^{1/2}\notag\\
&\leq \|z\|\phi(E((x_\lambda -x_0)^*(x_\lambda -x_0))).\label{eq3.12}
\end{align}
Since $x_\lambda \to x_0$ in the $B$-topology, this gives
\begin{equation}\label{eq3.13}
\lim_\lambda \psi(x_\lambda -x_0)=0,\ \ \ \ \psi\in L,
\end{equation}
and the uniform bound on $\|x_\lambda\|$ implies that $x_\lambda \to x_0$ in the $w^*$-topology from \eqref{eq3.13} and the norm density of $L$ in $N_*$. Thus $x_0=y_0$, as desired. \end{proof}
\begin{rem}\label{rem3.3}
(i) \quad Lemmas \ref{lem3.1}  and \ref{lem3.2} jointly imply that the $w^*$- and $B$-closures of norm bounded convex sets are equal, so differences in these topologies, at least for convex sets, appear only for unbounded sets. In principle, the $B$-topology on $N$ is dependent on the choice of the von Neumann subalgebra $M$ and of the normal conditional expectation onto it. These lemmas show, however, that for norm bounded convex sets the various possible $B$-topologies give the same closure. In the two situations where we will need the $B$-topology, the choices of subalgebras and conditional expectations will be canonical.

\medskip
 
(ii) \quad The converse of Lemma \ref{lem3.1} and the conclusion of Lemma \ref{lem3.2} for general $w^*$-closed sets are both false, as the following example shows. Let $N=L^\infty [0,1]$ and
$M=\mathbb{C}I$, and define $E:N\to M$ by
\begin{equation}\label{3.14}
E(f)=\left(\int_0^1 f(t)\,dt\right)I,\ \ \ \ f\in L^\infty [0,1].
\end{equation}
In this case, the $B$-topology on $L^\infty [0,1]$ is the restriction of the $\|\cdot\|_2$-norm topology on $L^2[0,1]$. Define a $w^*$-continuous functional $\phi$ on $L^\infty [0,1]$ by
integration against $t^{-1/2}\in L^1[0,1]\setminus L^2[0,1]$ and let $K$ be its kernel. Then $K$ is $\|\cdot\|_2$-dense in $L^2[0,1]$ since any nonzero vector orthogonal to $K$ would
have to be a multiple of $t^{-1/2}$ which is impossible. Thus, for example, the constant function 1 is not in the $w^*$-closed set $K$ but lies in its $B$-closure.
\end{rem}

\section{Weak$^*$-closed bimodules}\label{sec4}

In this section, we study the structure of certain $w^*$-closed bimodules arising in the crossed product setting.  In particular, given an outer action of a countable, discrete group $G$ on a factor $M$, any subset $S$ of $G$ gives rise to a $B$-closed $M$-bimodule 
\[  X_S = \ovl{{\mathrm{span}}}^B\sett{xg:x\in M,\ g\in S} \subseteq \cp. \]
It is natural to ask whether these are all the $B$-closed $M$-bimodules in the crossed product; that this is indeed the case is the main result of this section.

An important tool for our investigation here is a result of Christensen and Sinclair \cite{CS} which we state for the reader's convenience after establishing some notation. For a
von Neumann algebra $M$, let $L_{cb}(M,M)$ denote the space of completely bounded linear maps $\phi:M\to M$, while $L_{cb}(M,M)_M$ is the subspace of right $M$-module maps, those that
satisfy $\phi(xm)=\phi(x)m$ for $x,m\in M$. These of course have a very simple form, $\phi(x)=tx$ for $x\in M$ and an element $t=\phi(I)\in M$. Each collection $\beta=(m_j)_{j\in J}$
from $M$ satisfying $\sum_{j\in J}m_j^*m_j=I$ induces a map $\phi \mapsto\phi^\beta$ on $L_{cb}(M,M)$, where
\begin{equation}\label{eq4.1}
\phi^\beta(x)={\textstyle\sum_{j\in J}}\,\phi(xm_j^*)m_j,\ \ \ \ x\in M.
\end{equation}
The complete boundedness of $\phi$ guarantees that $\phi^\beta$ is well defined and satisfies $\|\phi^\beta\|_{cb}\leq \|\phi\|_{cb}$. For a fixed element $c\in M$, we denote by $\phi_c$
the map $\phi_c(x)=\phi(cx)$ for $x\in M$. In the following theorem, the first two parts constitute a special case of \cite[Theorem 3.3]{CS}, (see also \cite[Theorem 1.7.4]{SS}), while the third part is just an observation
that we will need subsequently.
\begin{thm}\label{thm4.1}
Let $M$ be a von Neumann algebra.
\begin{itemize}
\item[\rm (i)]
There exists a contractive projection $\rho:L_{cb}(M,M)\to L_{cb}(M,M)_M$.
\item[\rm (ii)]
There is a net $(\beta)$ so that
\begin{equation}\label{eq4.2}
\rho\phi(x)=w^*-\lim_\beta \phi^\beta(x),\ \ \ \ \phi\in L_{cb}(M,M),\ x\in M.
\end{equation}
\item[\rm (iii)]
If $c,d\in M$ and $\phi\in L_{cb}(M,M)$ satisfy
\begin{equation}\label{eq4.3}
\phi(cx)=d\phi(x),\ \ \ \ x\in M,
\end{equation}
then
\begin{equation}\label{eq4.4}
\rho\phi_c=d\rho\phi.
\end{equation}
\end{itemize}
\end{thm}
\begin{proof}
As noted above, only the last part requires proof, so suppose that $\phi\in L_{cb}(M,M)$ satisfies \eqref{eq4.3}. For each $\beta=(m_j)_{j\in J}$,
\begin{equation}\label{eq4.5}
{\textstyle\sum_{j\in J}}\,\phi_c(xm_j^*)m_j={\textstyle\sum_{j\in J}}\,d\phi(xm_j^*)m_j,\ \ \ \ x\in M,
\end{equation}
so using (ii) and taking $w^*$-limits in \eqref{eq4.5} gives \eqref{eq4.4} as required.
\end{proof}
We now apply this result to specific maps that will occur subsequently.
\begin{lem}\label{lem4.2}
Let $M$ be a von Neumann algebra, let $\alpha$ be a properly outer automorphism of $M$ and let $a\in M$ be fixed. If $\phi:M\to M$ is the completely bounded map defined by
\begin{equation}\label{eq4.6}
\phi(x)=\alpha(x)a,\ \ \ \ x\in M,
\end{equation}
then $\rho\phi=0$.
\end{lem}
\begin{proof}
Since $\rho\phi$ is a right $M$-module map, there exists $t\in M$ so that
\begin{equation}\label{eq4.7}
\rho\phi(x)=tx,\ \ \ \ x\in M.
\end{equation}
For each $y\in M$,
\begin{equation}\label{eq4.8}
\phi(yx)=\alpha(y)\alpha(x)a=\alpha(y)\phi(x),\ \ \ \ x\in M,
\end{equation}
so from Theorem \ref{thm4.1} (iii) we obtain
\begin{equation}\label{eq4.9}
\rho\phi_y(x)=\alpha(y)\rho\phi(x)=\alpha(y)tx,\ \ \ \ x\in M.
\end{equation}
 From Theorem \ref{thm4.1} (ii),
\begin{equation}\label{eq4.10}
\rho\phi_y(x)=\rho\phi(yx),\ \ \ \ x\in M,
\end{equation}
so
\begin{equation}\label{eq4.11}
\alpha(y)tx=tyx,\ \ \ x,y\in M.
\end{equation}
Put $x=1$ in \eqref{eq4.11} to obtain that $\alpha(y)t=ty$ for $y\in M$. This implies that $t=0$ since $\alpha$ is properly outer, and so $\rho\phi=0$ from \eqref{eq4.7}.
\end{proof}
If $G$ is a discrete group acting on a factor $M$ by outer automorphisms, then each $x\in \cp$ has a Fourier series $x=\sum_{g\in G}x_gg$, where convergence takes place in the $B$-topology
for the standard normal conditional expectation $E:\cp\to M$ \cite{Mer}. The $g$-coefficients $x_g\in M$ are given by $x_g=E(xg^{-1})$.
\begin{thm}\label{thm4.3}
Let a discrete group $G$ act on a factor $M$ by outer automorphisms $\alpha_g$, $g\in G$, and let $X\subseteq \cp$ be an intermediate $M$-bimodule which is either closed in the $B$-topology or the
$w^*$-topology. If $g_0\in G$ is such that there exists $x\in X$ with a nonzero $g_0$-coefficient, then $g_0\in X$.
\end{thm}
\begin{proof}
By Lemma \ref{lem3.1}, it suffices to consider the case when $X$ is $w^*$-closed. Suppose that $x\in X$ has a nonzero $g_0$-coefficient. Multiplying on the left by $x_{g_0}^*$, we may assume
that $x_{g_0}\geq 0$, and a further left multiplication by a suitable element of $M$ allows us to assume that $x_{g_0}$ is a nonzero projection $p$. Since $M$ is a factor, there exist
partial isometries $(v_k)_{k\in K}$ so that $v_k^*v_k\leq p$ and $\sum_{k\in K}v_kv_k^*=I$, so
$\sum_{k\in K}v_kx\alpha_{g_0^{-1}}(v_k^*)\in X$ (since $\alpha_{g_0^{-1}}$ is completely bounded), and the $g_0$-coefficient of this element is $I$. Thus we may assume that $x_{g_0}=I$.

Let $\beta=(m_j)_{j\in J}$ be the net from Theorem \ref{thm4.1}, pre- and post-multiply $x$ by $m_j^*$ and $\alpha_{g_0^{-1}}(m_j)$ and sum over $j$ to obtain elements $x^\beta\in X$ given by
\begin{equation}\label{eq4.12}
x^\beta ={\textstyle\sum_{j\in J}}\,m_j^*x\alpha_{g_0^{-1}}(m_j).
\end{equation}
The $g$-coefficient of $x^\beta$ is $\sum_{j\in J}m_j^*x_g\alpha_{gg_0^{-1}}(m_j)$. For $g=g_0$, this is $\sum_{j\in J}m_j^*m_j=I$, while for $g\ne g_0$ we may write this as
\begin{equation}\label{eq4.13}
\alpha_{gg_0^{-1}}\left({\textstyle\sum_{j\in J}}\,\alpha_{g_0g^{-1}}(m_j^*)\alpha_{g_0g^{-1}}(x_g)m_j\right).
\end{equation}
If we define $\phi^g\in L_{cb}(M,M)$ by
\begin{equation}\label{eq4.14}
\phi^g(t)=\alpha_{g_0g^{-1}}(t)\alpha_{g_0g^{-1}}(x_g),\ \ \ \ t\in M,
\end{equation}
then the $w^*$-limit over $\beta$ in \eqref{eq4.13} is $\alpha_{gg_0^{-1}}(\rho\phi^g(I))$, using the $w^*$-continuity of $\alpha_{gg_0^{-1}}$, and this is 0 by applying Lemma \ref{lem4.2}
to the outer automorphism $\alpha_{g_0g^{-1}}$ and the fixed element $a=\alpha_{g_0g^{-1}}(x_g)$. Since $\|x^\beta\|\leq \|x\|$, we may drop to a subnet and assume in addition that
$x^\beta\to y\in X$ in the $w^*$-topology. If $y$ has Fourier series $\sum_{g\in G}y_gg$, then
\begin{equation}\label{eq4.15}
y_g=E(yg^{-1})=\lim_\beta E(x^\beta g^{-1})
\end{equation}
by $w^*$-continuity of $E$ and so, from above, $y_g=I$ for $g=g_0$ and is 0 otherwise. Thus $g_0=y\in X$ as required.
\end{proof}
\begin{thm}\label{thm4.4}
Let $G$ be a discrete group acting by outer automorphisms $\alpha_g$, $g\in G$, on a factor $M$.
\begin{itemize}
\item[\rm (i)]
There is a bijective correspondence between subsets $S$ of $G$ and $B$-closed  $M$-bimodules of $\cp$ given by
\begin{equation}\label{eq4.16}
S\mapsto X_S := \ovl{{\mathrm{span}}}^B\sett{xg:x\in M,\ g\in S}.
\end{equation}
\item[\rm (ii)]
If $G$ has the $AP$,  then the collections of $B$-closed and $w^*$-closed bimodules coincide.
\end{itemize}
\end{thm}
\begin{proof}
(i) \quad Clearly the empty set corresponds to the $M$-bimodule $\sett{0}$. Now let $S_1$ and $S_2$ be distinct nonempty subsets of $G$ and suppose without loss of generality that $g_0\in S_1$ but that $g_0\notin S_2$. For a finite subset $F\subseteq S_2$, consider
a sum $\sum_{g\in F}x_gg$. Then
\begin{equation}\label{eq4.17}
E\left(\left(g_0-{\textstyle\sum_{g\in F}}\,x_gg\right)^*\left(g_0-{\textstyle\sum_{g\in F}}\,x_gg\right)\right)=I+
{\textstyle\sum_{g\in F}\,\alpha_{g^{-1}}(x_g^*x_g)}\geq I
\end{equation}
and so $g_0\notin \ovl{{\mathrm{span}}}^B\sett{xg:x\in M,\ g\in S_2}$. Thus the map of \eqref{eq4.16} is injective.

Suppose that $X$ is a $B$-closed $M$-bimodule. Then, by Theorem \ref{thm4.3}, $X$ contains each group element that appears with a nonzero coefficient in the Fourier series of some
element of $X$. Let $S$ denote the set of such elements. Then $S_X\subseteq X$. If $x\in X$ has Fourier series $\sum_{g\in G}x_gg$ then, restricting to the terms with nonzero
coefficients, we obtain that the finite partial sums lie in $X_S$ and converge to $x$ in the $B$-topology by \cite{Mer}. This shows that $X\subseteq X_S$, proving equality.

\medskip

\n (ii) \quad Suppose now that $G$ has the $AP$. By Lemma \ref{lem3.1}, each $B$-closed $M$-bimodule is $w^*$-closed so it is only necessary to prove the converse. Let $X\subseteq \cp$
be a $w^*$-closed $M$-bimodule. Let $S$ be the set of group elements that appear with a nonzero coefficient in the Fourier series of some element of $X$ and note that $S\subseteq X$ by
Theorem \ref{thm4.3}. Consider an element $y$ in the $B$-closure of $X$ with Fourier series $\sum_{g\in G}y_gg$, and suppose that $y_{g_0}\ne 0$ for some $g_0\in G\setminus S$. Choose a
normal state $\phi\in M_*$ such that $\phi(\alpha_{g_0}^{-1}(y_{g_0}^*y_{g_0}))>0$ and call this number $\vp >0$. Let $x=\sum_{g\in G}x_gg$ be an arbitrary element of $X$. By
construction of $S$, $x_{g_0}=0$, and so the $e$-coefficient of $(x-y)^*(x-y)$ is at least
\begin{equation}\label{eq4.18}
g_0^{-1}y_{g_0}^*y_{g_0}g_0=\alpha_{g_0}^{-1}(y_{g_0}^*y_{g_0}).
\end{equation}
Thus
\begin{equation}\label{eq4.19}
\phi(E((x-y)^*(x-y)))\geq \vp>0.
\end{equation}
Since $x\in X$ was arbitrary, this gives the contradiction that $y$ is not in the $B$-closure of $X$ and we conclude that the nonzero coefficients of $y$ can only occur for $g\in S$.

Now let $(f_\gamma)_{\gamma\in \Gamma}$ and $(T_\gamma)_{\gamma\in \Gamma}$ be the functions and operators of Lemma \ref{lem2.1}, let $z\in \cp$ be arbitrary, and let $x^\beta=\sum_{g\in
G}x_g^\beta g$ be a net of finitely supported elements converging to $z$ in the $w^*$-topology. Then, for each $g\in G$, 
$w^*-\underset{\beta}{\lim}\,  x_g^\beta =z_g$ since $x_g^\beta =E(x^\beta
g^{-1})$.
For each $\gamma\in \Gamma$,
\begin{equation}\label{eq4.20}
T_\gamma(x^\beta)={\textstyle\sum_{g\in G}}\, f_\gamma(g)x_g^\beta g
\end{equation}
from Lemma \ref{lem2.1} (vi), and since $T_\gamma(x^\beta)\to T_\gamma(z)$ in the $w^*$-topology, we also obtain that
\begin{equation}\label{eq4.21}
T_\gamma(z)={\textstyle\sum_{g\in G}}\, f_\gamma(g)z_gg,\ \ \ \ z\in \cp.
\end{equation}
Applying this to $y$ which is supported on $S$, we see that each $T_\gamma(y)$ is finitely supported on $S$ and thus lies in $X$. Since $T_\gamma(y) \to y$ in the $w^*$-topology, we
conclude that $y\in X$, completing the proof.
\end{proof}
 As an immediate consequence we have the following.
\begin{cor}\label{cor4.5}
Let $G$ be a discrete group acting by outer automorphisms $\alpha_g$, $g\in G$, on a factor $M$. The $B$-closed subalgebras of $\cp$ that contain $M$ are in bijective correspondence with
the unital semigroups $S\subseteq G$ by the map
\begin{equation}\label{eq4.22}
S\mapsto \ovl{{\mathrm{span}}}^B\sett{xg:x\in M,\ g\in G}.
\end{equation}
If, in addition, $G$ has the $AP$, then the sets of $B$-closed and $w^*$-closed subalgebras that contain $M$ coincide.
\end{cor}

If $H\subseteq G$ is a subgroup, then the C$^*$-subalgebra of $\cp$ generated by $M$ and $H$ is denoted by $M\rtimes_{\alpha,r}H$. This is the reduced C$^*$-algebra crossed product of $M$ by $H$.  
\begin{rem}\label{rem4.6}
If $N$ is a von Neumann algebra satisfying $M\subseteq N\subseteq \cp$ then, by Corollary \ref{cor4.5}, there is a subgroup $H$ of $G$ so that
\begin{equation}
M\rtimes_{\alpha,r}H\subseteq N\subseteq \ovl{N}^B=M\rtimes_{\alpha}H.
\end{equation}
Taking $w^*$-closures gives $N=M\rtimes_{\alpha}H$. This was first established in \cite{Cho} under the hypothesis that there is a normal conditional expectation of $\cp$ onto $N$, and
this assumption was removed in \cite{ILP}, although $M$ was required to have a separable predual in order to make use of special hyperfinite subfactors which may not exist in the general
case.$\hfill\square$
\end{rem}

\section{C$^*$-envelopes}\label{sec5}
In this section, $G$ is a discrete group acting by outer automorphisms $\alpha_g$, $g\in G$, on a factor $M$. Note that $M$ is not type I since such factors have no outer automorphisms.
We do not wish to place unnecessary restrictions on $M$, but some of the facts that we will require are only valid for $\sigma$-finite factors so we begin with a brief discussion of such
algebras.

Recall that a factor $M$ is $\sigma$-{\emph{finite}} if every set of orthogonal projections in $M$ is at most countable. This is equivalent to the existence of a faithful normal state $\phi$ on $M$,
\cite[Prop. II.3.19]{Tak1}. More generally, a projection $p\in M$ is $\sigma$-finite if $pMp$ is $\sigma$-finite. Using the characterization by faithful normal states, it is clear that any projection
dominated by a $\sigma$-finite projection is itself $\sigma$-finite. The same characterization also makes it obvious that countable orthogonal sums of $\sigma$-finite projections are
again $\sigma$-finite. If $p$ and $q$ are $\sigma$-finite projections, then $\sigma$-finiteness of $p\vee q$ follows from the general equivalence (see \cite[Prop. V.1.6]{Tak1})
\begin{equation}\label{eq5.1a}
(p\vee q)-q\sim p-(p\wedge q)
\end{equation}
showing that $p\vee q=(p\vee q-q)+q$ is an orthogonal sum of two $\sigma$-finite projections. More generally, if $(p_n)_{n\geq 1}$ is a countable set of $\sigma$-finite projections with
lattice supremum $p$, then $\sigma$-finiteness of $p$ follows by writing this projection as an orthogonal sum of $\sigma$-finite projections $(q_n)_{n\geq 1}$ where $q_1=p_1$ and, for
$n\geq 2$,
\begin{equation}\label{eq5.1b}
q_n=\vee\sett{p_i:1\leq i\leq n}-\vee\sett{p_i:1\leq i\leq n-1}.
\end{equation}
If $p$ is a $\sigma$-finite projection and $G$ is countable, then it is dominated by a $G$-invariant $\sigma$-finite projection, namely $\vee\sett{\alpha_g(p):g\in G}$. Our final
observation is that any nonzero projection $q$ dominates a nonzero $\sigma$-finite projection $p$: choose a normal state $\phi$ with $\phi(q)>0$, choose a maximal set of orthogonal
subprojections $(q_\lambda)_{\lambda\in\Lambda}$ with $\phi(q_\lambda)=0$, and set $p=q-\sum_{\lambda\in\Lambda}q_\lambda$.
We will be interested in the structure of norm closed ideals in factors and this can vary with cardinality. For example, the norm closed span of the $\sigma$-finite projections in a type
III factor $M$ is a norm closed ideal which will be proper if the cardinality is sufficiently large, in contrast to the simplicity of such factors on separable Hilbert spaces.

Our first objective is  Lemma \ref{lem5.1}, a technical lemma that constitutes part of the proof of the subsequent theorem and which will be used again in Section \ref{sec7}.

\begin{lem}\label{lem5.0}
Let $\alpha$ be a $\ast$-automorphism of a von Neumann algebra $Q$ and let $a,b,c,d\in Q$ satisfy
\begin{equation}\label{eq5.1}
axb=c\alpha(x)d,\ \ \ \ x\in Q.
\end{equation}
If there exists $x_0\in Q$ such that $ax_0b\ne 0$, then $\alpha$ is not properly outer.
\end{lem}
\begin{proof}
Since $(ax_0b)^*(ax_0b)=b^*x_0^*(a^*ax_0b)$, we see that $a^*ax_0b\ne 0$. This allows us to multiply on the left in \eqref{eq5.1} by $a^*$ so that we may assume that $a\geq
0$. Let $p_\vp$ be the spectral projection of $a$ for the interval $[\vp, \infty)$. For a sufficiently small choice of $\vp$, $p_\vp x_0b\ne0$ and so we may multiply on
the left in \eqref{eq5.1} by a suitable element which allows us to replace $a$ by a projection $p\in Q$. Thus we work with an equation of the form
\begin{equation}\label{eq5.2}
pxb_1=c_1\alpha(x)d_1,\ \ \ \ x\in Q,
\end{equation}
where $px_0b_1\ne 0$. As in \cite[p.61]{SS2}, we may find a family of partial isometries $v_i\in Q$ so that $v_i^*v_i\leq p$ and the projections $v_ipv_i^*$ are pairwise
orthogonal and sum strongly to the central support $c(p)$ of $p$. Then $p\leq c(p)$ so $c(p)x_0b_1\ne 0$ and
\begin{equation}\label{eq5.3}
xc(p)b_1=c(p)xb_1=\sum_iv_ipv_i^*xb_1
=\sum_i v_ic_1\alpha(v_i^*)\alpha(x)d_1,\ \ \ \ x\in Q,
\end{equation}
the sums converging strongly due to the complete boundedness of $\alpha$. This allows us to replace $p$ by 1 in \eqref{eq5.2} and to assume that we have an equation of
the form
\begin{equation}\label{eq5.4}
xb_2=c_2\alpha(x)d_2,\ \ \ \ x\in Q,
\end{equation}
where $x_0b_2\ne 0$. Then $c_2\alpha(x_0)d_2\ne 0$.

Repeat this argument on the right on both sides of \eqref{eq5.4} to replace $d_2$ by 1. The element $b_2$ will change, but we will have reduced to an equation of the
form
\begin{equation}\label{eq5.5}
xb_3=c_3\alpha(x),\ \ \ \ x\in M,
\end{equation}
where $x_0b_3\ne 0$. Setting $x=1$ gives $c_3=b_3\ne 0$, and we have shown that $\alpha$ is not free and hence not properly outer.
\end{proof}

\begin{lem}\label{lem5.1}
Let $Q$ be a von Neumann algebra  with a faithful normal semifinite trace $Tr$ and let $\alpha$ be a properly outer $\ast$-automorphism of $Q$. Let $p\in Q$ be a  projection,
let $N=\sett{p}'\cap Q$, and denote the unitary group of $N$ by $\mathcal U$. If $y\in Q\cap L^2(Q,Tr)$ and
$$K:={\mathrm{conv}}\,\sett{uy\alpha(u^*):u\in \mathcal U},$$
then $0\in \overline{K}^{\|\cdot\|_2}$.
\end{lem}

\begin{proof}
Since the elements of $K$ are bounded in the operator norm by $\|y\|$ and in the $\|\cdot\|_2$-norm by $\|y\|_2$, \cite[Lemma 9.2.1 (v)]{SS2} shows that
$\overline{K}^{\|\cdot\|_2}\subseteq Q\cap L^2(Q,Tr)$. Let $t\in Q\cap L^2(Q,Tr)$ be the element of minimal $\|\cdot\|_2$-norm in $\overline{K}^{\|\cdot\|_2}$. Then
\begin{equation}\label{eq5.6}
ut\alpha(u^*)=t,\ \ \ \ u\in \mathcal U,
\end{equation}
so, by taking linear combinations, we see that
\begin{equation}\label{eq5.7}
xt=t\alpha(x),\ \ \ \ x\in N.
\end{equation}
To derive a contradiction, suppose that $t\ne 0$. Then either $pt\ne 0$ or $p^\perp t\ne 0$, and since the argument is the same in both cases we assume without loss of
generality that  the first of these possibilities holds. Now $pQp,\ p^\perp Qp^\perp \subseteq N$, so substitution of elements of the first of these subalgebras into
\eqref{eq5.7} gives
\begin{equation}\label{eq5.8}
pxpt=t\alpha(p)\alpha(x)\alpha(p),\ \ \ \ x\in Q.
\end{equation}
(We would use the second subalgebra if $p^\perp t$ were nonzero). Since $pt\ne 0$, the hypotheses of Lemma \ref{lem5.0} hold with $a=p$, $b=pt$, $c=t\alpha(p)$,
$d=\alpha(p)$, and $x_0=1$. Thus $\alpha$ is not properly outer, and this contradiction shows that $t=0$ as required.
\end{proof}

\begin{lem}\label{lem5.2a}
Let $Q\subseteq P$ be an inclusion of von Neumann algebras with a faithful normal contractive conditional expectation $E:P\to Q$, and let $Q$ have a faithful normal semifinite trace $Tr$. Let $\Lambda$ be a set with a distinguished element $\lambda_0$ and let $\sett{v_\lambda:\lambda\in \Lambda}\subseteq \mathcal{N}(Q\subseteq P\}$ be a collection of normalizing unitaries such that $v_{\lambda_0}=1$ and the $\ast$-automorphisms $\alpha_\lambda\in {\mathrm{Aut}}(Q)$ given by
\begin{equation}\label{eq2001}
\alpha_\lambda(x)=v_\lambda xv_\lambda^*,\ \ \ \ \lambda\in \Lambda,\ x\in Q,
\end{equation}
are properly outer for $\lambda\ne\lambda_0$. Let $A$ be a $C^*$-algebra with the following properties:
\begin{itemize}
\item[\rm (i)] $Q\subseteq A\subseteq P$.
\item[\rm (ii)] For each $\lambda\in \Lambda$, there is an algebraic ideal $I_\lambda\subseteq Q$ with $I_{\lambda_0}=Q$ such that $A$ is the norm closure of ${\mathrm{span}}\,\sett{I_\lambda v_\lambda:\lambda\in \Lambda}$.
\end{itemize}
Let $J$ be a nonzero norm closed ideal in $A$. Then $Q\cap J\ne \sett{0}$.
\end{lem}
\begin{proof}
An application of $E$ to the equation
\begin{equation}\label{eq2002}
\alpha_\lambda(x)v_\lambda=v_\lambda x,\ \ \ \ x\in Q,
\end{equation}
shows that $E(v_\lambda)=0$ for $\lambda\ne \lambda_0$, otherwise one of these automorphisms would fail to be properly outer.

Let $j\in J^+$ be a nonzero positive element. By faithfulness, $E(j)\ne 0$, so multiplying by a suitable element of $Q$ allows us to assume that $E(j)$ is a nonzero finite projection $p\in Q$. Since finite sums of terms from $I_\lambda v_\lambda$ are norm dense in $A$, we may choose a finite sum ${\sum_{\lambda\in \Lambda}'}\,x_\lambda v_\lambda$ with $x_\lambda\in I_\lambda$ (where $\sum'$ denotes a finite sum) so that
\begin{equation}\label{eq5.9}
\left\|j-{\textstyle{\sum_{\lambda\in \Lambda}'}\,x_\lambda v_\lambda}\right\|<1/4.
\end{equation}
Applying the expectation $E$, we obtain $\|p-x_{\lambda_0}\|<1/4$, so if we add $(p-x_{\lambda_0}) $ to the finite sum then we have a new approximation 
\begin{equation}\label{eq5.10}
\left\|j-{	\textstyle{\sum_{\lambda\in \Lambda}'}\,x_\lambda v_\lambda}\right\|<1/2
\end{equation}
with the the additional property that $x_{\lambda_0}$ is now $p$. Multiplying on the left by $p$, we may further assume that each $x_\lambda$ lies in $L^2(M,Tr)$. 

Consider the set of all pairs $(j,\sum_{\lambda\in \Lambda}'\,x_\lambda v_\lambda)$, where $j\in J$, $E(j)$ is a finite nonzero projection $p$, $x_{\lambda_0}=p$, $x_\lambda\in I_\lambda\cap L^2(M,Tr)$ for $\lambda\in\Lambda$ and the inequality \eqref{eq5.10} is satisfied. We have just seen that there is at least one such choice, so from all the possibilities, select one for which the sum has minimal length, and call it $(j,\sum_{\lambda\in \Lambda}'\,x_\lambda v_\lambda)$. Our objective is to show that the length is 1, so to obtain a contradiction suppose that there is a $\lambda_1\ne \lambda_0$ with $x_{\lambda_1}\ne 0$.

Fix $\vp>0$ so that $\vp < Tr(p)$ and
\begin{equation}\label{eq5.11}
\left\|j-{\textstyle\sum'}_{\lambda\in \Lambda}\,x_\lambda v_\lambda\right\|+\sqrt{\vp}< 1/2.
\end{equation}
Since $\alpha_{\lambda_1}$ is an outer automorphism, we may apply Lemma \ref{lem5.1} to find a finite set of unitaries
$\sett{u_i\in N:=\sett{p}'\cap Q:1\leq i\leq k}$ and positive constants $\mu_i$ summing to 1 so that
\begin{equation}\label{eq5.12}
\left\|{\textstyle\sum_{i=1}^k} \,\mu_iu_ix_{\lambda_1}\alpha_{\lambda_1}(u_i^*)\right\|_2<\vp.
\end{equation}
Define a complete contraction $T:A \to A$ by
\begin{equation}\label{eq5.13}
T(y)={\textstyle\sum_{i=1}^k}\,\mu_ipu_iyu_i^*,\ \ \ \ y\in A.
\end{equation}
Then $T(j)\in J$, $E(T(j))$ is still $p$, and applying $T$ to \eqref{eq5.11} gives a new approximation
\begin{equation}\label{eq5.14}
\left\|T(j)-{\textstyle\sum_{\lambda\in \Lambda}'}\,z_\lambda v_\lambda\right\|+\sqrt{\vp}<1/2.
\end{equation}
The sum in \eqref{eq5.14} is still of minimal length, $z_{\lambda_0}=p$ since the $u_i$'s commute with $p$ and we have the additional features that $\|z_{\lambda_1}\|_2<\vp$ and $pz_{\lambda_1}=z_{\lambda_1}$. Let $q$ be the spectral projection
of $z_{\lambda_1}z_{\lambda_1}^*$ for the interval $[\vp,\infty)$. Then $\vp q\leq z_{\lambda_1}z_{\lambda_1}^*$ so
\begin{equation}\label{eq5.15}
\vp Tr(q)\leq Tr(z_{\lambda_1}z_{\lambda_1}^*)=\|z_{\lambda_1}\|_2^2<\vp^2,
\end{equation}
and thus
\begin{equation}\label{eq5.16}
Tr(q)<\vp< Tr(p).
\end{equation}
Note also that $q\leq p$, and that the projection $q_1=p-q$ is nonzero since $Tr(q_1)>0$ from \eqref{eq5.16}.
Since $q_1$ is orthogonal to $q$, we have   $\|q_1z_{\lambda_1}z_{\lambda_1}^*q_1\|\leq \vp$, so
$\|q_1z_{\lambda_1}\|\leq\sqrt{\vp}$. Returning to \eqref{eq5.14} and multiplying on the left by $q_1$, we obtain
\begin{equation}\label{eq5.17}
\left\|q_1T(j)-{\textstyle\sum_{\lambda\in \Lambda}'}\,q_1z_\lambda v_\lambda\right\|<1/2-\sqrt{\vp},
\end{equation}
and it follows that
\begin{equation}\label{eq5.18}
\left\|q_1T(j)-{\textstyle\sum_{\lambda\ne \lambda_1}'}\,q_1z_\lambda v_\lambda\right\|\leq \left\|q_1T(j)-{\textstyle\sum_{\lambda\in \Lambda}'}\,q_1z_\lambda v_\lambda\right\|+\left\|q_1z_{\lambda_1}\right\|< 1/2.
\end{equation}
Note that $q_1T(j)\in J$ has $E(q_1T(j))=q_1p=q_1$ and $q_1z_{\lambda_0}=q_1$. Consequently \eqref{eq5.18} gives a strictly shorter sum with the same properties, a contradiction. Thus the shortest such sum
has length 1, so there exists $j\in J$ and a nonzero projection $p\in Q$ such that $\|j-p\|< 1/2$. It follows that the restriction to $Q$ of the quotient map onto $A/J$ is not isometric, and so $J\cap Q\ne
\sett{0}$ as desired.
\end{proof}

Recall that $\rcp$ denotes the reduced C$^*$-algebra crossed product, a subalgebra of $\cp$ which is the norm closed span of the elements $mg$ for $m\in M$ and $g\in G$.
\begin{thm}\label{thm5.2}
Let $G$ be a discrete group acting on a factor $M$ by outer automorphisms $\alpha_g$, $g\in G$. Let $J$ be a norm closed nonzero ideal
 in $\rcp$. Then $J\cap M\ne \sett{0}$.
\end{thm}
\begin{proof}
We will prove this for countable groups and $\sigma$-finite factors and deduce the general case from this. We begin by handling the reduction, which we accomplish in two stages. Suppose that the result is true for countable groups and general factors.

Let $J$ be a nonzero norm closed ideal in $\rcp$. We may write $G$ as the union of an increasing net $(G_\lambda)_{\lambda\in \Lambda}$
of countable groups, and $\rcp$ is then the norm closure of $\cup_{\lambda \in \Lambda}\rcpl$. Since $J\ne \sett{0}$, it must have nontrivial intersection with $\rcpl$ for $\lambda$ sufficiently large and so $J\cap M\ne \sett{0}$ since each $G_\lambda$ is countable.

Now suppose that the result is true for countable groups and $\sigma$-finite factors, and consider a countable group $G$, a general factor $M$, and a nonzero norm closed ideal $J$ in $\rcp$. Let $j\in J$ be a nonzero element and write $j=\sum_{g\in G}x_gg$ for the Fourier series in $\cp$. At least one coefficient $x_{g_0}$ is nonzero, so multiplying by $x_{g_0}^*$ on the left and by $g_0^{-1}$ on the right allows us to assume that $x_e\geq 0$ and is nonzero. Multiplication on the left by a suitably chosen element of $M$ allows us to assume that $x_e$ is a nonzero projection $q$. Choose a nonzero $\sigma$-finite subprojection $p$ and a $G$-invariant $\sigma$-finite projection $p_1$ that dominates $p$. Then $pjp_1$ has Fourier series $\sum_{g\in G}px_gp_1g$ and is nonzero since the $e$-coefficient is $p$. Moreover, all coefficients lie in the $\sigma$-finite factor $p_1Mp_1$. A simple approximation argument gives $(\rcp)\cap (\cpcut)=\rcpcut$, and so $J\cap M\ne \sett{0}$ since we are assuming this for $\sigma$-finite factors. This completes the reduction and so it suffices to prove the result when $G$ is countable and $M$ is $\sigma$-finite and we henceforth assume that we are in this case, although this is only necessary for the type III situation.

If $M$ is type III and $K\subseteq M$ is a nonzero norm closed ideal, then spectral theory gives us a nonzero projection $p\in K$ and this is equivalent to $I-p$ since $M$ is $\sigma$-finite, \cite[Prop. V.1.39]{Tak1}. Then there is a partial isometry $v\in M$ so that $I-p=vpv^*$, and we obtain $I\in K$ and $K=M$. Thus $\sigma$-finite type III factors are simple, as are all type II$_1$ factors. In both cases $\rcp$ is simple, \cite{Ki}, and thus $J$ certainly contains $M$.

It remains to consider the case of a $\sigma$-finite type II$_\infty$ factor $M$ acted upon by a group $G$. Such a factor has the form $M_0\,\overline{\otimes}\,B(H)$ for a type II$_1$ factor $M_0$, and $H$ is separable otherwise $M$ would not be $\sigma$-finite. Thus $M$ has a faithful normal semifinite trace and Lemma \ref{lem5.2a} applies to the inclusion $M\subseteq \cp$. The result is then immediate from this lemma by taking $Q=M$, $P=\cp$, $\Lambda =G$, $\lambda_0=e$, $v_\lambda =g$, $I_\lambda =M$, and $A=\rcp$.\end{proof}

We now use these results to show that $\rcp$ is the C$^*$-envelope of certain operator subspaces of this C$^*$-algebra. Recall that a C$^*$-algebra $A$ is said to be the C$^*$-{\emph{envelope}}
of a unital operator space $X$ if there is a completely isometric unital embedding $\iota:X\to A$ so that $\iota(X)$ generates $A$, and if $B$ is another C$^*$-algebra with a completely
isometric unital embedding $\iota':X\to B$ whose range generates $B$, then there is a *-homomorphism $\pi:B\to A$ so that $\pi\circ\iota'=\iota$ (which entails surjectivity of $\pi$).
Every unital operator space has a unique C$^*$-envelope denoted C$^*_{env}(X)$, \cite{Ar1,Ar2,Ha}.
\begin{thm}\label{thm5.3}
Let $G$ be a discrete group acting on a factor $M$ by outer automorphisms $\alpha_g$, $g\in G$, let $S$ be a unital subset of $G$ which generates $G$, and let $X\subseteq \rcp$ be an
operator space that contains $\sett{mg:m\in M,\  g\in S}$. Then C$^*_{env}(X)=\rcp$.
\end{thm}
\begin{proof}
Let $\iota:X\to C^*_{env}(X)$ be a completely isometric unital embedding and let $\iota':X\to \rcp$ be the identity embedding, whose range generates $\rcp$. From the definition, there
exists a surjective *-homomorphism $\pi:\rcp \to C^*_{env}(X)$ so that $\pi\circ\iota'=\iota$. If $\pi$ were not a $\ast$-isomorphism, then it would have a nontrivial kernel $J$. By Theorem
\ref{thm5.2}, $J\cap M$ would contain a nonzero element $m$, giving the contradiction $\iota(m)=\pi(m)=0$. Thus $\pi$ is a $\ast$-isomorphism, proving the result.
\end{proof}

\section{Extension of maps}\label{sec6}
In this section we will consider $w^*$-closed $M$-bimodules $M\subseteq
X\subseteq\cp$ and isometric $w^*$-continuous maps $\theta : X\to \cp$ which respect the
modular structure in the sense that the restriction of $\theta$ to $M$ is a
$\ast$-automorphism and
\begin{equation}\label{eq6.1}
\theta (m_1xm_2)=\theta(m_1)\theta(x)\theta(m_2),\ \ \ \ m_1,m_2\in M,\ x\in X.
\end{equation}

Our objective is to show that such maps extend to $\ast$-automorphisms of $\cp$
when $X$ generates $\cp$. We will require some preliminary lemmas. The first is well known to experts, but we include it for completeness.
\begin{lem}\label{lem6.0}
Let $G$ be a discrete group acting on a factor $M$ by outer automorphisms $\alpha_g$, $g\in G$. Let $w\in \cp$ be a unitary that normalizes $M$. Then there exist $g_0\in
G$ and a unitary $u\in M$ so that $w=ug_0$.
\end{lem}

\begin{proof}
Let $\theta$ be the $\ast$-automorphism of $M$ defined by $\theta(x)=w^*xw$, $x\in M$, and let $w=\sum_{g\in G}w_gg$ be the Fourier series of $w$. Then
\begin{equation}\label{eq6.01}
\sum_{g\in G}xw_gg=\sum_{g\in G}w_g\alpha_g(\theta(x))g,\ \ \ x\in M.
\end{equation}
Thus
\begin{equation}\label{eq6.02}
xw_g=w_g\alpha_g(\theta(x)),\ \ \ x\in M,
\end{equation}
so $\alpha_g\circ \theta$ is inner whenever $w_g\ne 0$. If $g_0\ne g_1$ are elements of $G$ for which $w_{g_0}$ and $w_{g_1}$ are both nonzero, then
$\alpha_{g_0}\circ\theta$ and $\alpha_{g_1}\circ\theta$ are both inner, implying that
$\alpha_{g_0g_1^{-1}}=\alpha_{g_0}\circ\theta\circ\theta^{-1}\circ\alpha_{g_1^{-1}}$ is also inner, contradicting the hypothesis that this is outer. Thus there exists
precisely one $g_0\in G$ so that $w_{g_0}\ne 0$, and so $w=w_{g_o}g_0$ and $w_{g_0}$ is then seen to be a unitary, completing the proof.
\end{proof}

\begin{lem}\label{lem6.1}
Let $G$ be a discrete group acting on a factor $M$ by outer automorphisms $\alpha_g$, $g\in G$.
Let $\theta$ be a $\ast$-automorphism of $\rcp$ satisfying $\theta(M)=M$. Then
$\theta$ extends to a $\ast$-automorphism of $\cp$.
\end{lem}
\begin{proof}
For each $g\in G$, $\theta(g)$ is a unitary normalizer of $M$ and so, by Lemma \ref{lem6.0}, has the
form $w_g\rho(g)$ where $w_g$ is a unitary in $M$ and $\rho$ is a
permutation of $G$ with $\rho(e)=e$. Consider an element $x=\sum_{g\in
G}x_gg\in \cp$  and choose, by the Kaplansky density
theorem, a uniformly bounded net  $(x^\beta)_{\beta\in B}$ from
${\mathrm{span}}\,\sett{yg:y\in M,\ g\in G}$ converging to $x$ in the
$w^*$-topology. Drop to a subnet if necessary to assume that the net
$(\theta(x^\beta))_{\beta\in B}$ converges in the $w^*$-topology to an
element $y\in \cp$. If we write $x^\beta = \sum_{g\in
G}x_g^\beta g$, a finite sum, then $\theta(x^\beta)=\sum_{g\in
G}\theta(x^\beta_g)w_g\rho(g)$ and so the $\rho(g)$-coefficient of $y$ is
\begin{equation}\label{eq6.2}
w^*-\lim_{\beta}\theta(x_g^\beta)w_g=\theta(x_g)w_g
\end{equation}
since $\theta|_M$ is $w^*$-continuous. In particular, $y$ is independent of
the choice of approximating net. Thus there is a well-defined map
$\phi:\cp\to\cp$ given by
\begin{equation}\label{eq6.3}
\phi\left({\textstyle\sum_{g\in G}}\,x_gg\right)={\textstyle\sum_{g\in G}}\,\theta(x_g)\theta(g),\ \ \ \ x={\textstyle\sum_{g\in
G}}\,x_gg\in \cp.
\end{equation}
It is clear from its definition that $\phi$ is linear. If $\|x\|\leq 1$, then the net $(x^\beta)_{\beta\in B}$ could have been chosen to be a net of contractions,
showing that $\|\phi(x)\|\leq 1$. Thus $\phi$ is a linear contraction.

By applying the same argument to $\theta^{-1}$, we find that $\phi$ has a
contractive inverse  and is thus a surjective isometry.

We now turn to the $w^*$-continuity of $\phi$. By the Krein-Smulian
theorem, in proving $w^*$-continuity of a linear map, it suffices to
consider a uniformly bounded net $(x^\beta)_{\beta\in B}$ converging to 0 in
the $w^*$-topology. Since the net $(\phi(x^\beta))_{\beta\in B}$ is
uniformly bounded, let $y$ be any $w^*$-limit of a subnet
$(\phi(x^\gamma))_{\gamma\in \Gamma}$. Then, writing $x^\gamma=\sum_{g\in
G}x^\gamma_gg$, we see that $w^*-\underset{\gamma}{\lim}\, x^\gamma_g=0$ for each $g\in
G$, so the $w^*$-limit of the $\rho(g)$-coefficients of $\phi(x^\beta)$ is
0. Thus $y=0$ and it follows that $\phi(x^\beta)\to 0$ in the
$w^*$-topology. Having shown that $\phi$ is $w^*$-continuous, we immediately
conclude that it is a $\ast$-automorphism of $\cp$.
\end{proof}
\begin{lem}\label{lem6.2}
Let $G$ be a discrete group acting on a factor $M$ by outer automorphisms $\alpha_g$, $g\in G$.
Let $M\subseteq X\subseteq \cp$ be a $w^*$-closed $M$-bimodule and let $\theta: X\to X$
be a $w^*$-continuous isometric surjective isomorphism so that $\theta$ is a
$\ast$-automorphism of $M$ and
\begin{equation}\label{eq6.4}
\theta(m_1xm_2)=\theta(m_1)\theta(x)\theta(m_2),\ \ \ \ m_1,m_2\in M,\ \ x\in X.
\end{equation}
If $x\in X$ has Fourier series $\sum_{g\in G}x_gg$, then $\theta(x)$ has
Fourier series $\sum_{g\in G}\theta(x_g)\theta(g)$.
\end{lem}
\begin{proof}
Let $S$ be the set of group elements that appear in the Fourier series of elements of $X$ with a nonzero coefficient. By Theorem \ref{thm4.3}, $\sett{mg:m\in M,\ g\in S}\subseteq X$. If $g\in S$,
then
\begin{equation}\label{eq6.5}
gm=\alpha_g(m)g,\ \ \ \ m\in M,
\end{equation}
so
\begin{equation}\label{eq6.6}
\theta(g)\theta(m)=\theta(\alpha_g(m))\theta(g),\ \ \ \ m\in M.
\end{equation}
This implies that $\theta(g)\theta(g)^*$ and $\theta(g)^*\theta(g)$ lie in $M'\cap(\cp)=\mathbb{C}I$, and since these elements have norm 1 we conclude that $\theta(g)$ is a unitary.
Moreover, by \eqref{eq6.6}, $\theta(g)$ normalizes $M$ so by Lemma \ref{lem6.0}, it has the form $u_g\rho(g)$ for a unitary $u_g\in M$ and a permutation $\rho$ of $S$, with $\rho(e)=e$ and $u_e=I$.
Consequently $\sum_{g\in G}\theta(x_g)\theta(g)$ is indeed a Fourier series although it is not expressed in the customary way.

We argue by contradiction, so suppose that there are elements $x\in X$ and $g_0\in S$ so that $x$ has Fourier series $\sum_{g\in G} x_gg$ but the $\rho(g_0)$-term in the Fourier series
of $\theta(x)$ is not $\theta(x_{g_0})\theta(g_0)$. We first consider the special case where $g_0=e$. By subtracting $x_e$ from $x$, we may assume that $x$ has $e$-coefficient 0 but
$\theta(x)$ has a nonzero $e$-coefficient. Multiplying on the left by a suitably chosen element of $M$, this $e$-coefficient may be taken to be a nonzero projection $p\in M$. Choose
partial isometries $(v_\lambda)_{\lambda\in \Lambda}$ so that $\sum_{\lambda\in \Lambda}\,v_\lambda pv_\lambda^*=I$ and let $F\subseteq \Lambda$ be a finite subset. For each $z\in \cp$, define a contraction $T_F$ by $T_F(z)=\sum_{\lambda\in
F}\theta^{-1}(v_\lambda)z\theta^{-1}(v_\lambda^*)$. The net $(T_F(x))$, indexed by the finite subsets of $\Lambda$, is uniformly bounded so, by dropping to a subnet if necessary, we may assume convergence in the
$w^*$-topology to an element $y\in X$ whose $e$-coefficient is 0. By $w^*$-continuity of $\theta$, the net $(\theta(T_F(x)))$ converges in the $w^*$-topology to $\theta(y)$ whose
$e$-coefficient is $\sum_{\lambda\in \Lambda}v_\lambda pv_\lambda^*=I$. Replacing $x$ by $y$ if necessary, we may make the further assumption that the $e$-coefficient of $\theta(x)$ is 1.

 From Section \ref{sec4}, let $\beta=(m_j)_{j\in J}$ be the set 
of collections of operators satisfying $\sum_{j\in J}m_j^*m_j=I$ and such that Theorem \ref{thm4.1} holds. Define a complete contraction $R_{\beta}:\cp\to\cp$ by
\begin{equation}\label{eq6.7}
R_\beta (z)={\textstyle\sum_{j\in J}}m_j^*zm_j,\ \ z\in \cp,
\end{equation} where the sum converges in the $w^*$-topology. Then $R_\beta$ maps $X$ to itself since $X$ is $w^*$-closed, and the $g$-coefficient of $R_\beta(z)$ for $z=\sum_{g\in
G}z_gg$ is $\sum_{j\in J}m_j^*z_g\alpha_g(m_j)$. Thus $R_\beta(x)$ has $e$-coefficient 0 while the $g$-coefficients for $g\ne e$ tend to 0 in the $w^*$-topology over the net $(\beta)$ as
in the proof of Theorem \ref{thm4.3}. By dropping to a subnet, we may assume that $w^*-\underset{\beta}{\lim}\, R_\beta(x)$ exists, and this must be 0. Thus $w^*-\underset{\beta}{\lim}\,  \theta(R_\beta(x))=0$. From
\eqref{eq6.7},
\begin{equation}\label{eq6.8}
\theta(R_\beta(x))={\textstyle\sum_{j\in J}}\theta(m_j^*)\theta(x)\theta(m_j)
\end{equation}
and the $e$-coefficient in this sum is 1. Thus $\theta(R_\beta(x))$ cannot converge to 0 in the $w^*$-topology, a contradiction.

Returning to the general case, suppose that the desired formula fails at the $\rho(g_0)$-term. Define a $w^*$-closed $M$-bimodule by $Y=Xg_0^{-1}$ and define $\theta':Y \to Y$ by
\begin{equation}\label{eq6.9}
\theta'(xg_0^{-1})=\theta(x)\theta(g_0)^*,\ \ \ \ x\in X.
\end{equation}
This map is isometric since $\theta(g_0)$ is a unitary. Now, for $m\in M$,
\begin{equation}\label{eq6.10}
\theta'(m)=\theta'(mg_0g_0^{-1})=\theta(mg_0)\theta(g_0)^*=\theta(m)\theta(g_0)
\theta(g_0)^*=\theta(m),
\end{equation}
while, for $x\in X$ and $m\in M$,
\begin{equation}\label{eq6.11}
\theta'(mxg_0^{-1})=\theta(mx)\theta(g_0)^*=
\theta(m)\theta(x)\theta(g_0)^*=\theta'(m)\theta'(xg_0^{-1}).
\end{equation}
For multiplication on the right, first note that
\begin{equation}\label{eq6.12}
g_0\alpha_{g_0}^{-1}(m)=mg_0,\ \ m\in M,
\end{equation}
yielding
\begin{equation}\label{eq6.13}
\theta(\alpha_{g_0}^{-1}(m))=\theta(g_0)^*\theta(m)\theta(g_0),\ \ m\in M.
\end{equation}
Thus
\begin{align}
\theta'(xg_0^{-1}m)&=\theta'(x\alpha_{g_0}^{-1}(m)g_0^{-1})
=\theta(x\alpha_{g_0}^{-1}(m))\theta(g_0)^*
=\theta(x)\theta(\alpha_{g_0}^{-1}(m))\theta(g_0)^*\notag\\
&=\theta (x)\theta(g_0)^*\theta(m)
=\theta'(xg_0^{-1})\theta'(m),\ \ m\in M.\label{eq6.14}
\end{align}  
This puts us into the case of failure at $g_0=e$ for the pair $(Y,\theta')$ and we have already shown that this cannot happen.
\end{proof}
We come now to the main result of the section for which we will need the concept of a norming subalgebra \cite{PSS}. Let $A\subseteq B$ be an inclusion of C$^*$-algebras. If $X$ is an $n\times n$ matrix over $B$ and $R$ and $C$ are respectively rows and columns of length $n$ over $A$ of unit norm, then
\begin{equation}\label{eq6.15}
\|RXC\|\leq \|X\|.
\end{equation}
If the supremum of the left hand side over rows and columns of unit norm equals $\|X\|$ for every matrix $X$ of any size, then we say that $A$ norms $B$. This has proved useful in settling issues of automatic complete boundedness of bounded maps, and will be used below.
\begin{thm}\label{thm6.3}
Let $M\subseteq X\subseteq \cp$ be a $w^*$-closed $M$-bimodule that generates $\cp$, let $\theta:X\to X$ be a $w^*$-continuous 
surjective isometry satisfying
\begin{equation}\label{eq6.16}
\theta(m_1xm_2)=\theta(m_1)\theta(x)\theta(m_2),\ \ m_1,m_2\in M,\ x\in X,
\end{equation}
and suppose that the restriction of $\theta$ to $M$ is a $\ast$-automorphism. Then $\theta$ extends uniquely to a $\ast$-automorphism of $\cp$.
\end{thm}
\begin{proof}
Let $S$ be the set of $g\in G$ so that $g$ appears with a nonzero coefficient in the Fourier series of an element of $X$. Let
\begin{equation}\label{eq6.17}
Y=\overline{{\mathrm{span}}}^{\|\cdot\|}\sett{mg:m\in M,\ g\in S},
\end{equation}
which is contained in $X$ by Theorem \ref{thm4.3}. For each $g\in S$, $\theta(g)$ is a normalizer of $M$ and consequently $\theta$ maps $Y$ to itself. When $M$ is type II$_{\infty}$ or III, it
norms $\cp$, \cite{PSS}, while the same conclusion is reached for type II$_1$ factors in \cite{PS}. Thus $X$ and $Y$ are normed by $M$ and so $\theta$ is a complete isometry. Since $X$
generates $\cp$, $S$ must generate $G$ and so $Y$ must generate $\rcp$ as a C$^*$-algebra. By Theorem \ref{thm5.3}, the C$^*$-envelope of $Y$ is $\rcp$ and so, from the theory of
C$^*$-envelopes, the embedding $\theta$ of $Y$ into $\rcp$ extends to a $\ast$-automorphism $\phi$ of $\rcp$, and thence to a $\ast$-automorphism of $\cp$, also denoted by $\phi$ (see Lemma
\ref{lem6.1}). For $g\in G$ and $m\in M$, we have
\begin{equation}\label{eq6.18}
\phi(mg)=\theta(mg)=\theta(m)\theta(g).
\end{equation}
For each $x\in X$ with Fourier series $\sum_{g\in G}x_gg$, Lemma \ref{lem6.2} gives $\theta(x)=\sum_{g\in G}\theta(x_g)\theta(g)$. If we apply this lemma again with $\cp$ in place of $X$ and
$\phi$ replacing $\theta$, then we obtain $\phi(z)=\sum_{g\in G}\phi(z_g)\phi(g)$ for each $z=\sum_{g\in G}z_gg\in \cp$. Since $\theta$ and $\phi$ agree on $Y$, we now see that they
agree on $X$. Thus we have found an extension to $\cp$, and uniqueness is immediate since any two extensions will agree on $X$ which contains a generating set for $\cp$.
\end{proof}
\begin{rem}
Under the notation and hypotheses of Theorem \ref{thm6.3}, if $X$ happens to be an intermediate subalgebra, then $\theta$ is the restriction of a $\ast$-automorphism of $\cp$ and so is
automatically an algebraic isomorphism of $X$. It is perhaps surprising that such a conclusion follows from the assumption of just being an isometry.
 \end{rem}

\section{Regular subfactors}\label{sec7}

In this section we investigate the situation of a regular inclusion $M\subseteq N$ of II$_1$ factors, where {\emph{regular}} means that $N$ is generated by the group
$\mathcal{N}(M\subseteq N)$ of unitaries in $N$ that normalize $M$. In particular, we will consider whether analogous results to those obtained in Section \ref{sec6}
hold for a $w^*$-closed $Q$-bimodule $X$ with $Q\subseteq X\subseteq N$, where $Q$ is the von Neumann algebra generated by $M$ and its relative commutant $C:=M'\cap N$.
We will state below a structural result for such inclusions, but first we establish some notation and prove two preliminary lemmas.

If we let $L$ denote the group generated by the commuting groups $\mathcal{U}(M)$ and $\mathcal{U}(C)$, then $L$ is a normal subgroup of $\mathcal{N}(M\subseteq N)$ and
there is a short exact sequence
\begin{equation}\label{eq6001a}
1\rightarrow L\rightarrow \mathcal{N}(M\subseteq N)\rightarrow G \rightarrow 1
\end{equation}
where $G$ is the quotient group $\mathcal{N}(M\subseteq N)/L$ which we view as a discrete group. The quotient map has a cross section $g\mapsto u_g$, $g\in G$, which is
defined pointwise by choosing representatives of the cosets, and we always make the choice that $u_e=1$. For each pair $g,h\in G$, the unitaries $u_gu_h$ and $u_{gh}$
differ by an element of $L$, and so there is an $L$-valued 2-cocycle $\omega(g,h)$ on $G\times G$ such that
\begin{equation}\label{eq6001b}
u_gu_h=\omega(g,h)u_{gh},\ \ \ \ g,h\in G.
\end{equation}
In general, it is only possible to choose a homomorphic cross section when $\omega$ is a 2-coboundary. Each $u_g$ induces a $\ast$-automorphism Ad$\,u_g$ on $Q$ which we
denote by $\alpha_g$, and \eqref{eq6001b} gives
\begin{equation}\label{eq6001c}
\alpha_g\alpha_h={\text{Ad}}\,(\omega(g,h))\alpha_{gh},\ \ \ \ g,h\in G,
\end{equation}
so that the map $g\mapsto \alpha_g$ is not generally a homomorphism of $G$ into Aut$(Q)$. The next two lemmas set out the properties of the $u_g$'s and $\alpha_g$'s. The first of these is due to Kallman \cite{Kal}.  We include a short proof for the reader's convenience. We use the notation $E_P$ to denote the trace preserving conditional expectation of $N$ onto a von Neumann subalgebra $P$. Note that $E_P$ is a $P$-bimodule map.

\begin{lem}\label{lem7.1a}
Let $M\subseteq N$ be an inclusion of II$_1$ factors and let $\tau$ be the trace on $N$.  Let
$C=M'\cap N$ and let $Q=W^*(M,C)$. If $\theta$ is a $\ast$-automorphism of $Q$ which restricts to an outer $\ast$-automorphism of $M$, then $\theta$ is properly outer on
$Q$.
\end{lem}
\begin{proof}
Suppose that $\theta$ is not properly outer on $Q$ and choose a nonzero element $q\in Q$ such that
\begin{equation}\label{eq6001}
qx=\theta(x)q,\ \ \ \ x\in Q.
\end{equation}
Since products of elements from $M$ and $C$ span a $w^*$-dense subspace of $Q$, we may choose $m\in M$ and $c\in C$ so that $\tau(mcq)\ne 0$. Then $E_M(mcq)\ne 0$ and the
same is true for $E_M(cq)$ since $mE_M(cq)=E_M(mcq)$. Noting that $c$ commutes with $\theta(x)$ for $x\in M$, we may multiply \eqref{eq6001} by $c$ on the left and apply $E_M$
to reach
\begin{equation}\label{eq6002}
E_M(cq)x=\theta(x)E_M(cq),\ \ \ x\in M.
\end{equation}
This contradicts the outerness of $\theta$ on $M$, so $\theta $ is properly outer on $Q$.
\end{proof}

\begin{lem}\label{lem7.1b} 
Let $M\subseteq N$ be an inclusion of II$_1$ factors, let $C=M'\cap N$, let $Q=W^*(M,C)$ and let $L=\mathcal{U}(M)\mathcal{U}(C)$. Let
$G=\mathcal{N}(M\subseteq N)/L$ and let $g\mapsto u_g\in \mathcal{N}(M\subseteq N)$ be a cross section for $G$ with $u_e=1$.
\begin{itemize}
\item[\rm (i)] For $g\in G\setminus \sett{e}$, $\alpha_g:={\mathrm{Ad}}\,u_g$ is a properly outer $\ast$-automorphism of $Q$ and an outer $\ast$-automorphism of $M$.
\item[\rm (ii)] $\alpha_g\alpha_h^{-1}$ is properly outer on $Q$ if and only if $g\ne h$.
\item[\rm(iii)] For $g\in G\setminus\sett{e}$, $E_M(u_g)=0$ and $E_Q(u_g)=0$.
\end{itemize}
\end{lem}
\begin{proof} (i)\quad Suppose that $g\ne e$. If $\alpha_g$ is inner on $M$ then there exists $v\in \mathcal{U}(M)$ such that $\alpha_g={\text{Ad}}\,v$ on $M$. Then
$v^*u_g\in M'\cap N=C$, so $v^*u_g$ can be written as a unitary $w\in \mathcal{U}(C)$. Thus $u_g=vw\in L$. This is a contradiction, and so $\alpha_g$ is outer on $M$. It then follows from Lemma \ref{lem7.1a}
 that it is also properly outer on $Q$.

\medskip

\noindent (ii)\quad Since there are unitaries $v,w\in L$ such that $u_h^{-1}=vu_{h^{-1}}$ and $u_gu_{h^{-1}}=wu_{gh^{-1}}$, $\alpha_g\alpha_h^{-1}$ differs from
$\alpha_{gh^{-1}}$ by an inner automorphism on $Q$, and the result follows from (i).

\medskip

\noindent (iii)\quad If we apply $E_M$ to the equation $u_gx=\alpha_g(x)u_g$ for $x\in M$, the result is $E_M(u_g)x=\alpha_g(x)E_M(u_g)$. It follows from (i) that $E_M(u_g)=0$
for $g\ne e$, otherwise $\alpha_g$ would be inner on $M$. Similarly $E_Q(u_g)=0$, otherwise the equation 
$E_Q(u_g)x=\alpha_g(x)E_Q(u_g)$ for $x\in Q$ would contradict the proper outerness of $\alpha_g$ on $Q$, by part (i).
\end{proof}

The {\emph{twisted crossed product}} $Q\rtimes_{\alpha,\omega}G$ is defined to be the von Neumann algebra generated by $Q$ and the set of normalizers $\sett{u_g:g\in G}$. Such
algebras were studied in \cite{MCho}. The connection to regular inclusions of subfactors is exhibited by the following basic structural result from \cite{Cam}, in the
spirit of the Feldman-Moore theory of Cartan masas \cite{FM1,FM2}.

 \begin{lem}[Theorem 4.6 of \cite{Cam}] \label{lem:cp} Let $M \subseteq N$ be a regular inclusion of II$_1$ factors, and let $Q$ denote the von Neumann algebra generated by $M$ and $M' \cap N.$  Then there exists a  discrete group $G$, and a 2-cocycle $\omega: G \times G \rightarrow \U(Q)$ such that $N = Q \rtimes_{\alpha, \omega} G$.
\end{lem}

This result was used in \cite{Cam} to show that, if $M \subseteq N$ is a regular inclusion of II$_1$ factors, then $M$ norms $N$.  Note that in this situation the von Neumann subalgebra $Q$ is spatially isomorphic to the tensor product $M \ovltimes (M' \cap N).$

Throughout the remainder of this section, we maintain the following notation: $M \subseteq N$ is a regular inclusion of II$_1$ factors, $C = M' \cap N,$ and $Q = M \ovltimes C$. Then, as in Lemma \ref{lem:cp} we have that $N = Q \rtimes_{\alpha, \omega} G$ for some discrete group $G$. We use  $X$ to denote a $w^*$-closed $Q$-bimodule with $Q \subseteq X \subseteq N$.  

It follows from Lemma \ref{lem7.1b} and the regularity of the inclusion $M\subseteq N$ that the subspaces $\sett{Qu_g:g\in G}$ are mutually orthogonal in $L^2(N)$ and that this space is the direct sum of the $\|\cdot\|_2$-closures of these subspaces. Thus each $x\in N$ has a Fourier series $x=\sum_{g\in G} x_gu_g$ where $x_g\in Q$ is given by $x_g=E_Q(xu_g^*)$, just as in the crossed product case. This series converges in $\|\cdot\|_2$-norm.

\begin{lem} \label{lemCoeff}
For each $g \in G,$ the set  $J_g = Q \cap X u_{g}^*$ is a $w^*$-closed ideal in $Q$, and if $x \in X$ has Fourier series 
$\sum_{g \in G} x_g u_g,$ then $x_g \in J_g$ for each $g \in G$.  
\end{lem}

\begin{proof}  

It is clear that $J_g$ is a $w^*$-closed ideal in $Q$.  Fix $g_0 \in G$ and consider $x = \sum_{g\in G}  x_g u_g \in X.$  Multiply by $x_{g_0}^*$, so that $x_{g_0}^*x = \sum_{g\in G} x_{g_0}^* x_g u_g \in X$.  If $x_{g_0}^* x_{g_0} \in J_{g_0}$ then so is $x_{g_0}$, so it suffices to assume that $x_{g_0} \geq 0.$  

For each $\eps > 0$ let $p_\eps$ be the spectral projection of $x_{g_0}$ for $[ \eps, \infty).$  Then $x_{g_0} p_\eps$ is invertible in $p_\eps Q p_\eps,$ so multiplying by $(x_{g_0} p_\eps)^{-1},$ we may further replace $x_{g_0}$ by $p_\eps$.  Thus we may assume that $x_{g_0}$ is a projection $p$.  If $z$ is its central support then there are partial isometries $v_i$ so that $z = \sum_i v_i p v_i^*$, so multiplying on the left by $v_i$ and on the right by $\alpha_{g_0}^{-1}(v_i^*),$ and summing, we can then make $x_{g_0}$ the central projection $z$.  Indeed if $z \in J_{g_0},$ then also $p = pz \in J_{g_0},$ so it suffices to take $x_{g_0} = z.$  Now for each unitary $u \in Q,$ we have
$u x \alpha_{g_0}^{-1}(u^*) = \sum_{g\in G} u x_g \alpha_g(\alpha_{g_0}^{-1}(u^*)) u_g \in X$.  Averaging over the unitary group of $Q$, we obtain an element $h \in X$ whose $u_g$-coefficient is the element $h_g$ of minimal $\norm{ \cdot}_2$ in the weak closure of  ${\mathrm{conv}} \sett{u x_g \alpha_g(\alpha_{g_0}^{-1} (u^*)): u \in \U(Q)},$ which satisfies
\begin{equation} u h_g = h_g \alpha_g(\alpha_{g_0}^{-1} (u^*)) \end{equation}
for all $u \in \U(Q).$  By proper outerness of the action of $G$, we must have $h_g = 0$ unless $g = g_0,$ in which case $h_{g_0}= x_{g_0} = z.$  Then  $z u_{g_0} \in X,$ so $z \in J_{g_0}$ as required.  \end{proof}

\begin{lem} \label{lemGenerate} 
With notation as in Lemma \ref{lemCoeff}, let $Y$ be the linear span of the spaces $ \sett{J_g u_g: g \in G}$.  If $X$ generates $N$ as a von Neumann algebra, then so does $Y$.  
\end{lem}

\begin{proof}
First let $X_1$ and $X_2$ be subspaces of $N$, and suppose that there exist subspaces $Y_1$ and $ Y_2$ of $N$ so that 
$X_i$ is contained in the $\norm{\cdot}_2-$norm closure in $N$ of $Y_i$, $i = 1,2.$  We claim that $X_1 X_2$ is contained in the closure of $Y_1 Y_2$ in this topology.   Let $x_1, x_2 \in X$ and, given $\eps > 0$, pick $y_1 \in Y$ so that 
$\norm{x_1 - y_1}_2 < \eps/(2\norm{x_2}).$  Then choose $y_2 \in Y$ so that $\norm{x_2 - y_2} < \eps/(2 \norm{y_1}).$ We then have 
\begin{equation} \norm{x_1 x_2 - y_1 y_2}_2 \leq \norm{x_1 x_2 - y_1 x_2}_2 + \norm{y_1 x_2 - y_1 y_2} < \eps.\end{equation}
Extending to finite sums of such products, we see that $X_1 X_2 \subseteq \overline{Y_1 Y_2}^{\norm{\cdot}_2} \cap N.$  This argument extends by induction to finite products of subspaces $X_i \subseteq \overline{Y_i}^{\norm{\cdot}_2} \cap N,$   $1 \leq i \leq k,$ so that  $X_1 X_2 \cdots X_k \subseteq \overline{Y_1 Y_2 \cdots Y_k}^{\norm{\cdot}_2} \cap N.$

We apply the above when each $X_i$ is $X$ or $X^*$ and each $Y_i$ is $Y$ or $Y^*.$  If $x \in X$ has Fourier series $x = \sum_{g\in G} x_g u_g,$ then by Lemma \ref{lemCoeff}, we have $x_g u_g \in Y$ for each $g\in G$.  Since this series converges in $\norm{\cdot}_2$-norm, we see that $x \in  \overline{Y}^{\norm{\cdot}_2},$  so $X \subseteq \overline{Y}^{\norm{\cdot}_2}.$  By the above argument, it follows that ${\mathrm{Alg}}(X,X^*)$ is contained in 
$\overline{{\mathrm{Alg}}(Y,Y^*)}^{\norm{\cdot}_2} \cap N$.  Since $X$ generates $N$ as a von Neumann algebra, this says 
that ${\mathrm{Alg}}(Y,Y^*)$ is $\norm{\cdot}_2$-dense in $N$, and so $Y$ generates $N$.  
\end{proof}

\begin{lem} \label{lemenvelope}
With the notation of Lemma \ref{lemGenerate}, the C$^*$-envelope of $Y$ is C$^*(Y)$.
\end{lem}

\begin{proof}  Let $J$ be a nonzero ideal in the C$^*$-algebra C$^*(Y)$.  We may apply Lemma \ref{lem5.2a} with $A={\mathrm{C}}^*(Y)$ to conclude that $J \cap Q \neq \sett{0}.$ 

Now let $i: Y \rightarrow C^*_{env}(Y)$ be a completely isometric unital embedding and let $i': Y \rightarrow C^*(Y)$ be the identity embedding.  Then there is a surjective $\ast$-homomorphism $\pi: C^*(Y) \rightarrow C^*_{env}(Y)$ such that 

\begin{equation} \pi \circ i'(y) = i(y), \quad y \in Y.\end{equation}

Let $J = \ker \pi$.  If $J$ is a nonzero ideal then it has a nonzero intersection with $Q$, so choose $q \in Q$ with $q \neq 0$ and $\pi(q) = 0.$  Then also $i(q) = 0,$ a contradiction.  \end{proof}

We now turn to the main question of this section, in which we consider $w^*$-closed bimodules $Q \subseteq X \subseteq N$ (in the notation above), and isometric $w^*$-continuous maps $\theta: X \rightarrow X$  which restrict to $\ast$-automorphisms of $Q$ that fix the subfactor $M$.  We aim to show that any such map extends to a $\ast$-automorphism of the containing II$_1$ factor $N$.  Notice that this is a generalization of our work above on crossed products, in the special case that the containing factor is of type II$_1.$

\begin{lem}\label{lemY}
Let $X$ be a $w^*$-closed $Q$-bimodule with $Q \subseteq X \subseteq N$, and let $Y$ be as in Lemma \ref{lemGenerate}.  Let $\theta: X \rightarrow X$ be a $w^*$-continuous, surjective isometry whose restriction to $Q$ is a $\ast-$automorphism such that $\theta(M) = M$ and 

\begin{equation} 
\theta(q_1 x q_2) = \theta(q_1) \theta(x) \theta(q_2), \ \ \ \ q_1,q_2 \in Q, \ x \in X.
\end{equation}
Then $\theta$ maps $Y$ onto $Y$.  Moreover, $\theta$ is completely isometric.
\end{lem}

\begin{proof}
For each $g_0\in G$, $J_{g_0}$ is a $w^*$-closed ideal in $Q$, so is $Qp$ for some central projection $p\in Z(Q)$. Since $\theta(ju_{g_0})=\theta(j)\theta(pu_{g_0})$ for $j\in J_{g_0}$, it suffices to prove that $\theta(z u_{g_0}) \in Y$ whenever $z \in Z(Q)$ and $z u_{g_0} \in X.$  Applying $\theta$ to the equation

\begin{equation} q z u_{g_0} = z q u_{g_0} = z u_{g_0} \alpha_{g_0}^{-1}(q), \ \ \ \  q \in Q,\end{equation}

we obtain

\begin{equation} \theta(q) \theta(z u_{g_0}) = \theta(z u_{g_0}) \theta[ \alpha_{g_0}^{-1}(q)], \ \ \ \  q \in Q.\end{equation}

Let $\sum_{g\in G} y_g u_g$ be the Fourier series of $\theta(z u_{g_0})$.  Then, computing coefficients, we have 

\begin{equation} \theta(q) y_g = y_g \alpha_g[ \theta ( \alpha_{g_0}^{-1}(q))], \ \ \ \ q \in Q, \ g \in G.\end{equation}

Replacing $q$ by $\theta^{-1}(q),$ we get

\begin{equation}\label{eqxyz} q y_g = y_g ( \alpha_g \circ \theta \circ \alpha_{g_0}^{-1} \circ \theta^{-1})(q) \ \ \ \  q \in Q.\end{equation}

We claim that $\theta(z u_{g_0})$ has at most one nonzero Fourier coefficient $y_g$.  Suppose that there are distinct group elements $g_1$ and $g_2$ for which $y_{g_1}$ and $y_{g_2}$ are nonzero.  Then the automorphisms

\begin{equation} \psi_1 = \alpha_{g_1} \circ \theta \circ \alpha_{g_0}^{-1} \circ \theta^{-1} \ \ \text{and} \ \ \psi_2 = \alpha_{g_2} \circ \theta \circ \alpha_{g_0}^{-1} \circ \theta^{-1} \end{equation}

\noindent of $Q$ both fail to be properly outer on $Q$ from \eqref{eqxyz}.  As Q is generated by $M$ and its relative commutant, and by our hypotheses both $M$ and $M' \cap N$ are invariant under $\theta$, by Lemma \ref{lem7.1a} this implies that both automorphisms are inner on $M$.    Thus, the composition

\begin{equation} \psi_1 \circ \psi_2^{-1} = \alpha_{g_1} \circ \alpha_{g_2}^{-1} = {\mathrm{Ad}}\, \omega(g_1, g_2^{-1}) \circ \alpha_{g_1 g_2^{-1}} \end{equation}

\noindent is an inner automorphism of $M$.  This is a contradiction.  Thus, the Fourier series of $\theta(z u_{g_0})$ has only one nonzero coefficient, so $\theta(z u_{g_0})$ lies in $Y$.  To see that $\theta$ maps $Y$ onto itself, replace $q_i$ by $\theta^{-1}(q_i),$ $i = 1,2$, in the equation 

\begin{equation} \theta(q_1 x q_2) = \theta(q_1) \theta(x) \theta(q_2), \ \ \ \ q_1,q_2 \in Q, \ x \in X \end{equation}

\noindent and apply $\theta^{-2}$ to see that 

\begin{equation} \theta^{-1}(q_1 x q_2) = \theta^{-1}(q_1) \theta^{-1}(x) \theta^{-1}(q_2), \ \ \ \ q_1,q_2 \in Q, \ x \in X.\end{equation}

\noindent Thus, $\theta^{-1}$ satisfies the same properties that $\theta$ was assumed to have, so we may repeat the argument above to see that $\theta^{-1}$ also maps $Y$ into itself.  This completes the proof that $\theta$ maps $Y$ onto itself.

To prove that $\theta$ is a complete isometry, note that $\theta$ is a surjective isometry from $X$ onto itself.  The subalgebra $Q$ of $X$ is norming \cite{Cam}, so we may use the $Q$-bimodularity properties of $X$ and $\theta$, and follow the proof   of 
\cite[Theorem 1.4]{Pitts} to show that $\theta$ is a complete contraction.  The same argument can be used for $\theta^{-1}: X \rightarrow X$, so $\theta$ is a complete isometry from $X$ onto itself.   \end{proof}

Since $\theta: Y \rightarrow Y$ is a surjective, complete isometry, by the universal property of $C^*$-envelopes and Lemma \ref{lemenvelope}, $\theta$ extends to a $\ast$-automorphism $\phi$ of $C^*(Y)$.  Our next objective is to  show that this map extends further to a $\ast$-automorphism of $N$, for which we will need the following lemma, analogous to Lemma \ref{lemY}, and for which we require some notation.
 For $g\in G$, let $K_g\subseteq Q$ be the set of elements $j$ so that $ju_g\in {\mathrm{Alg}}(Y,Y^*)$. Then $K_g$ is an algebraic ideal in $Q$ and, by examining finite products from $Y$ and $Y^*$, we see that $K_g$ is also the set of elements which occur as the $u_g$-coefficients in Fourier series of elements of ${\mathrm{Alg}}(Y,Y^*)$. Since this algebra is $w^*$-dense in $N$, we also see that $K_g$ is $w^*$-dense in $Q$ for each $g\in G$. Now let $I_g$ denote the norm closure of $K_g$, which is an ideal in $Q$. Approximating elements of C$^*(Y)$ by elements from  ${\mathrm{Alg}}(Y,Y^*)$ shows that $I_g$ is also the set of elements which occur as the $u_g$-coefficients in the Fourier series of elements of  C$^*(Y)$ and $ju_g\in {\text{C}}^*(Y)$ for all $j\in I_g$.
The conclusion of the next result is similar to that of the previous one, but
requires different methods since the ideals $I_g$ may not be $w^*$-closed in contrast to the ideals $J_g$.

\begin{lem} \label{lemCstarext}
With the above notation, for each $g\in G$ there exists a unitary $w_g\in Q$ such that
\begin{equation}
\phi(ju_g)=\theta(j)w_gu_{\sigma(g)},\ \ \ \ j\in I_g,
\end{equation}
where $\sigma$ is an automorphism of $G$. 
\end{lem}

\begin{proof}
We first establish that all elements $\phi(ju_g)$, $g\in G$, $j\in I_g$, have Fourier series of length at most one, so suppose that this fails for particular choices $k_0\in G$ and $j_0\in I_{k_0}$. For each $j\in I_{k_0}$ write the Fourier series of $\phi(ju_{k_0})$ as 
\begin{equation}\label{eq1002}
\phi(ju_{k_0})=\sum_{g\in G}x_{j,g}u_g.
\end{equation}
By assumption there are two distinct elements $g_1,g_2\in G$ such that $x_{j_0,g_i}\ne 0$ for $i=1,2$. By scaling $j_0$, we may assume that both have norm at least one. 

Now, for $j\in I_{k_0}$,
\begin{equation}\label{eq1003}
ju_{k_0}xu_{k_0}^*j^*=j\alpha_{k_0}(x)j^*,\ \ \ \ x\in Q,
\end{equation}
so multiplication on the right by $ju_{k_0}$ leads to 
\begin{equation}\label{eq1004}
ju_{k_0}x\alpha_{k_0}^{-1}(j^*j)=j\alpha_{k_0}(x)j^*ju_{k_0},\ \ \ \ x\in Q.
\end{equation}
Apply $\phi$ in \eqref{eq1004} and use \eqref{eq1002} to reach 
\begin{equation}\label{eq1005}
\sum_{g\in G} x_{j,g}u_g\phi(x)\phi(\alpha_{k_0}^{-1}(j^*j))=\sum_{g\in G}\phi(j)\phi(\alpha_{k_0}(x))\phi(j^*)x_{j,g}u_g,
 \ \ \ \ x\in Q,\ j\in I_{k_0}.
\end{equation}
Comparison of coefficients in \eqref{eq1005} gives
\begin{equation}\label{eq1006}
x_{j,g}\alpha_g(\phi(x))\alpha_g(\phi(\alpha_{k_0}^{-1}(j^*j)))=\phi(j)\phi(\alpha_{k_0}(x))\phi(j^*)x_{j,g},\ \ \ \ x\in Q,\ j\in I_{k_0},\ g\in G.
\end{equation}
If we define $\ast$-automorphisms of $Q$ by $\beta_g=\alpha_g\circ \phi \circ \alpha_{k_0}^{-1}\circ \phi^{-1}$, $g\in G$, and make the substitution $y=\phi(\alpha_{k_0}(x))$, then \eqref{eq1006} becomes
\begin{equation}\label{eq1007}
\phi(j)y\phi(j^*)x_{j,g}=x_{j,g}\beta_g(y)\alpha_g(\phi(\alpha_{k_0}^{-1}(j^*j))),\ \ \ \ y\in Q,\ j\in I_{k_0},\ g\in G.
\end{equation}

Choose an increasing contractive positive approximate identity $\sett{e_\lambda:\lambda\in \Lambda}$ for $I_{k_0}$. Since $I_{k_0}$ is $w^*$-dense in $Q$, we have  the additional property that 
\begin{equation}\label{eq1008}
{\mathrm{SOT}}-\lim_\lambda e_\lambda =1.
\end{equation}

Consider now $r,j\in I_{k_0}$.
Then
\begin{equation}\label{eq1009}
\sum_{g\in G}x_{rj,g}u_g=\phi(rju_{k_0})=\phi(r)\phi(ju_{k_0})=\sum_{g\in G}\phi(r)x_{j,g}u_g,
\end{equation}
so comparison of coefficients in \eqref{eq1009} leads to
\begin{equation}\label{eq1010}
\phi(r)x_{j,g}=x_{rj,g},\ \ \ \ r,j\in I_{k_0},\ g\in G.
\end{equation}
Similarly, for $j,s\in I_{k_0}$,
\begin{align}
\sum_{g\in G}x_{js,g}u_g&=\phi(jsu_{k_0})=\phi(ju_{k_0}\alpha_{k_0}^{-1}(s))\notag\\
&=\sum_{g\in G}x_{j,g}u_g\phi(\alpha_{k_0}^{-1}(s))
=\sum_{g\in G}x_{j,g}\alpha_g(\phi(\alpha_{k_0}^{-1}(s)))u_g,\label{eq1011}
\end{align}
so
\begin{equation}\label{eq1012}
x_{js,g}=x_{j,g}\alpha_g(\phi(\alpha_{k_0}^{-1}(s))),
\ \ \ \ j,s\in I_{k_0},\ g\in G.
\end{equation}
Replacing $r$ by $e_\lambda^3$ in \eqref{eq1010} gives
\begin{equation}\label{eq1013}
\phi(e_\lambda^3)x_{j,g}=\phi(e_\lambda^2)x_{e_\lambda j,g}
=\phi(e_\lambda^2)x_{e_\lambda,g}\alpha_g(\phi(\alpha_{k_0}^{-1}(j))),\ \ \ \ \lambda\in \Lambda,\ g\in G,\ j\in I_{k_0},
\end{equation}
where the second equality follows from \eqref{eq1012} with $e_\lambda$ and $ j$ in place of $j$ and $s$ respectively.

 From \eqref{eq1008} we may choose $\lambda_0$ so large that
\begin{equation}\label{eq1014}
\|\phi(e_\lambda^3)x_{j_0,g_i}\|\geq 1/2,\ \ \ \
\lambda \geq\lambda_0,\ i=1,2.
\end{equation}

It follows from \eqref{eq1013} that
\begin{equation}\label{eq1016}
\|\phi(e_\lambda^2)x_{e_\lambda,g_i}\|
\geq 1/2,\ \ \ \ \lambda\geq\lambda_0,\ i=1,2,
\end{equation}
and in particular that all of these elements are nonzero. Returning to \eqref{eq1007} and replacing $j$ by $e_\lambda$, $g$ by $g_i$, and $y$ by $x$, we see that
\begin{equation}\label{eq1017}
\phi(e_\lambda)x\phi(e_\lambda)x_{e_\lambda,g_i}=x_{e_\lambda,g_i}
\beta_{g_i}(x)\alpha_{g_i}(\phi(\alpha_{k_0}^{-1}(e_\lambda^2 ))),\ \ \ \ x\in Q,\ \lambda\geq\lambda_0,
\ i=1,2.
\end{equation}
Then \eqref{eq1017} has the form of \eqref{eq5.1} with $a=\phi(e_\lambda)$ and $b=\phi(e_\lambda)x_{e_\lambda,g_i}$ and it follows from \eqref{eq1016} that the hypotheses of
Lemma \ref{lem5.0} are met with $x_0=1$. By that lemma, we conclude that $\beta_{g_1}$ and $\beta_{g_2}$ are not properly outer on $Q$ and so their restrictions to $M$ are both
inner by Lemma \ref{lem7.1a}. Thus
$\alpha_{g_1}\circ \alpha_{g_2}^{-1}=\beta_{g_1}\circ \beta_{g_2}^{-1}$ is inner on $M$, in contradiction to Lemma \ref{lem7.1b} since $g_1\ne g_2$. This proves that $\phi(ju_g)$ has a Fourier
series of length at most one whenever $ju_g\in {\text{C}}^*(Y)$.

For a particular $g\in G$, if there exist
nonzero elements  $j_1,j_2\in I_g$ so that $\phi(j_iu_g)=k_iu_{h_i}$ for $i=1,2$ with $h_1\ne h_2$, then the Fourier series of $\phi((j_1+j_2)u_g)$ has length two,
contradicting what we have already proved. Thus there is a map $\sigma:G\to G$ and linear contractions $\mu_g:I_g\to Q$, $g\in G$, so that
\begin{equation}\label{eq1018}
\phi(ju_g)=\mu_g(j)u_{\sigma(g)},\ \ \ \ g\in G,\ j\in I_g.
\end{equation}
Fix an arbitrary $g\in G$ and let $\sett{f_\lambda:\lambda\in \Lambda}$ be a contractive positive approximate identity for $I_g$ with the property of \eqref{eq1008}. Since
$\sett{\mu_g(f_\lambda):\lambda\in \Lambda}$ is a set of contractions, we may drop to a subnet and assume that there is an element $w_g\in Q$ such that $w^*-\underset{\lambda}{\lim}\, 
\mu_g(f_\lambda)=w_g$. Since $\underset{\lambda}{\lim}\, \|jf_\lambda -j\|=0$ for each $j\in I_g$, \eqref{eq1018} gives
\begin{equation}\label{eq1019}
\phi(ju_g)=\lim_\lambda\phi(jf_\lambda u_g)=\lim_\lambda
\phi(j)\mu_g(f_\lambda)u_{\sigma(g)}
=\phi(j)w_gu_{\sigma(g)},\ \ \ \ j\in I_g.
\end{equation}

It remains to show that $\sigma$ is an automorphism of $G$ and that each $w_g$ is a unitary. Take $g,h\in G$ and choose increasing positive contractive approximate identities $\sett{e_{\lambda,g}:\lambda\in
\Lambda_g}$ and $\sett{e_{\nu,h}:\nu\in \Lambda_h}$ for $I_g$ and $I_h$ respectively, both tending strongly to 1. Then the equation
\begin{equation}\label{eq1020}
\phi(e_{\lambda,g}\alpha_g(e_{\nu,h})u_gu_h)=\phi(e_{\lambda,g} u_g
e_{\nu,h} u_h)=\phi(e_{\lambda,g} u_g)\phi(e_{\nu,h} u_h)
\end{equation}
shows that $\phi(e_{\lambda,g} \alpha_g(e_{\nu,h})u_gu_h)$ lies simultaneously in $Qu_{\sigma(g)\sigma(h)}$ and $Qu_{\sigma(gh)}$. Sufficiently large choices of $\lambda$ and $\nu$
ensure that $e_{\lambda,g} \alpha_g(e_{\nu,h})\ne 0$, and thus $\sigma(gh)=\sigma(g)\sigma(h)$, proving that $\sigma$ is a homomorphism.

Since $w_g$ is a contraction, it suffices to show that $w_gw_g^*=1$; the finiteness of $Q$ will then establish that $w_g^*w_g=1$. If $w_gw_g^*\ne 1$, then there exist
$\vp >0$ and a spectral projection $p$ for $w_gw_g^*$ such that $\|pw_g\|\leq 1-\vp$. Then $e_{\lambda,g} \phi^{-1}(p)\in I_g$ for $\lambda\in \Lambda_g$, so
\begin{equation}\label{eq1021}
\phi(e_{\lambda,g} \phi^{-1}(p)u_g)=\phi(e_{\lambda,g})pw_gu_{\sigma(g)}.
\end{equation}
Thus
\begin{equation}\label{eq1022}
\|\phi(e_{\lambda,g}\phi^{-1}(p)u_g)\|\leq \|pw_g\|\leq 1-\vp,\ \ \ \ \lambda\in \Lambda_g.
\end{equation}
On the left hand side, the limit over $\lambda$ is 1 since $e_{\lambda,g}\to 1$ strongly, and this contradiction proves that each $w_g$ is a unitary.

The above arguments apply equally to $\phi^{-1}$ and so there exist unitaries $v_g\in Q$ for $g\in G$ and a homomorphism $\rho:G\to G$ such that
\begin{equation}\label{eq1023}
\phi^{-1}(ju_g)=\phi^{-1}(j)v_gu_{\rho(g)},
\ \ \ \ g\in G,\ j\in I_g.
\end{equation}
Computing $ju_g$ as both $\phi^{-1}(\phi(ju_g))$ and $\phi(\phi^{-1}(ju_g))$ leads easily to the conclusion that $\rho\sigma$ and $\sigma\rho$ are both the identity map
on $G$. This shows that $\sigma$ is an automorphism and completes the proof.
\end{proof}

\begin{lem}\label{lem7.8}
The $\ast$-automorphism $\phi$ of C$^*(Y)$ extends to a $\ast$-automorphism of $N$.
\end{lem}
\begin{proof}
Lemma \ref{lemCstarext} establishes that
\begin{equation}
\phi(ju_g)=\phi(j)w_gu_{\sigma(g)}
\end{equation}
for $j\in I_g$, $g\in G$, an automorphism $\sigma$ of $G$ and unitaries $w_g\in Q$. Based on this, the remainder of the proof follows that of Lemma \ref{lem6.1} with suitable changes of notation. We omit the details.
\end{proof}

Denote the extension of the $\ast$-automorphism $\phi$ to all of $N$ by $\overline{\theta}$.  We know that this map coincides with the original isometric bimodule map $\theta$ on $Y$, and we will have proved our main result -- that $\overline{\theta}$ is the unique extension of $\theta$ to a $\ast$-automorphism of $N$ -- if we can show that the two maps coincide on the ultraweakly closed bimodule $X$.  To do this, we first prove the following lemma.

\begin{lem} \label{lemPhi} If $x \in X$ has Fourier series $x = \sum_{g \in G} x_g u_g,$ then $\theta(x)$ has Fourier series $\sum_{g \in G} \theta(x_g u_g)$.
\end{lem}

\begin{proof}
If $x \in Q$ and $g \in G$ are such that $x u_g \in Y$, then $x$ lies in the ideal $J_g = X u_g^{-1} \cap Q$ of Lemma \ref{lemCoeff}, and thus in the larger ideal $I_g$ of Lemma \ref{lemCstarext}. By this last lemma, there is an automorphism $\sigma$ of $G$ such that $\theta$ maps  $J_gu_g$ into $I_{\sigma(g)}u_{\sigma(g)}$, and since $\theta $ maps $Y$ to $Y$ we see that $\theta$ takes $J_gu_g$ into 
$J_{\sigma(g)}u_{\sigma(g)}$.  The same observations applied to $\theta^{-1}$ also show that $\theta^{-1}(J_{\sigma(g)} u_{\sigma(g)}) \subseteq J_g u_g$ for  $g \in G.$  This establishes that $\theta(J_g u_g) = J_{\sigma(g)} u_{\sigma(g)}.$

To prove the result, proceed by contradiction and suppose that we have an element  $x = \sum_{g\in G} x_g u_g \in X$ for which $\theta(x) = \sum_{g \in G} y_g u_g \in X$ with $y_{\sigma(g_0)} u_{\sigma(g_0)} \neq \theta(x_{g_0} u_{g_0})$ for some $g_0 \in G.$  Subtracting  $y_{\sigma(g_0)} u_{\sigma(g_0)}$ from both sides of the equation 

\begin{equation}
\theta(x) = \sum_{g\in G} y_g u_g, 
\end{equation}

\noindent and noting that $y_{\sigma(g_0)} u_{\sigma(g_0)}$ is the image under $\theta$ of some element of $J_{g_0} u_{g_0}$ distinct from $x_{g_0} u_{g_0}$,  we obtain an element $x = \sum_{g\in G} x_g u_g \in X$ with $x_{g_0} \neq 0$, for which $\theta(x) = \sum_{g\in G} y_g u_g$ has $y_{\sigma(g_0)}=0$.

By further reductions similar to those in Lemma \ref{lemY}, we may also assume that $x_{g_0}$ is a nonzero central projection in $Q$.  Note that these operations do not change the fact that $y_{\sigma(g_0)} = 0.$  We now use familiar averaging techniques to pick out individual terms in the respective Fourier series.

For any unitary $w \in Q,$ we have 

\begin{equation} \theta(w) \theta\left( \sum_{g\in G} x_g u_g \right) \theta( \alpha_{g_0}^{-1}(w^*)) = \theta \left(\sum_{g \in G} w x_g \alpha_{g}\alpha_{g_0}^{-1}(w^*) u_g\right), \end{equation}

\noindent by the $Q$-bimodularity property of $\theta.$  On the other hand, 

\begin{equation} \theta(w) \sum_{g \in G} y_g u_g \theta (\alpha_{g_0}^{-1}(w^*)) = \sum_{g \in G} \theta(w) y_g \alpha_g \circ \theta \circ \alpha_{g_0}^{-1}(w^*) u_g .\end{equation}

\noindent This says that for any $w \in \U(Q),$ the operator  $\theta \left(\sum_{g \in G} w x_g \alpha_{g}\circ\alpha_{g_0}^{-1}(w^*) u_g\right)$ lies in the weak closure of the convex set 

\[ K = {\mathrm{conv}} \sett{ \sum_{g \in G} v y_g \alpha_g \circ \theta \circ \alpha_{g_0}^{-1} \circ \theta ^{-1}(v^*) u_g: v \in \U(Q)}.\]

\noindent Then since $\theta$ is $w^*$-continuous, applying $\theta$ to the element of minimal $\norm{ \cdot}_2$-norm in the weak closure of the convex set

\[ K' = {\mathrm{conv}} \sett{ \sum_{g \in G} w x_g \alpha_{g}\alpha_{g_0}^{-1}(w^*) u_g: w \in \U(Q)} \]

\noindent also gives an element of the weak closure of $K$.  By the methods used in Lemma \ref{lemCoeff}, we know that the element $h$ of minimal $\norm{\cdot}_2$-norm in the weak closure of $K'$ has the form $h = x_{g_0} u_{g_0},$ so this says that $\theta(x_{g_0} u_{g_0}) \in K ,$ and has the form $q_{\sigma(g_0)} u_{\sigma(g_0)}$ for some $q_{\sigma(g_0)} \in Q.$  But since $y_{\sigma(g_0)} = 0,$ every element of $K$ has zero $\sigma(g_0)-$coefficient, so this says that $\theta(x_{g_0} u_{g_0})= 0,$ a contradiction to the fact that $\theta$ is isometric on $X$.   This proves the desired result.     \end{proof}

We come now to the main result of this section, our version of Mercer's theorem for regular inclusions of subfactors.

\begin{thm} \label{thmExtension}
Let $M \subseteq N$ be a regular inclusion of II$_1$ factors.  Denote by $Q$ the von Neumann algebra generated by $M$ and $M' \cap N,$ and let $X$ be a w$^*$-closed $Q$-bimodule with $Q \subseteq X \subseteq N$, which generates $N$ as a von Neumann algebra.  Suppose that $\theta: X \rightarrow X$ is a w$^*$-continuous, surjective isometry whose restriction to $Q$ is a $\ast$-automorphisms with $\theta(M) = M$ and which satisfies

\[ \theta(q_1 x q_2) = \theta(q_1) \theta(x) \theta(q_2),\]

\noindent for all $x \in X$ and $q_1, q_2 \in Q.$  Then $\theta$ has a unique extension to a $\ast$-automorphism $\overline{\theta}: N \rightarrow N.$  

\end{thm}

\begin{proof}  Let $\overline{\theta}$ be the extension to $N$ of the $\ast$-automorphism $\phi: C^*(Y) \rightarrow C^*(Y)$ obtained in Lemma \ref{lemCstarext}.  We know that $\overline{\theta}$ coincides with $\theta$ on the subspace $Y \subseteq N$, and so we are done if we can show that these two maps coincide on $X$.  By  Lemma \ref{lemPhi}, we have that for all $x = \sum_{g \in G} x_g u_g \in X,$ 

\begin{equation} \theta(x) = \sum_{g \in G} \theta(x_g u_g),\end{equation}

\noindent and since $\overline{\theta}$ and $\theta$ coincide on $Y$, the same argument also shows that 

\begin{equation} \overline{\theta}(x) = \sum_{g \in G} \overline{\theta}(x_g u_g) \end{equation}

\noindent for all such $x.$  Thus, since $\theta(x_g u_g) = \overline{\theta}(x_g u_g)$ for each $g$, we see that  $\theta(x) = \overline{\theta}(x)$.  \end{proof}

\section{Single generation}\label{sec8}
A longstanding problem  is the question of whether each separable von Neumann algebra is singly generated. There has been considerable work on this problem \cite{DP,DSSW,GS,P,Top,W} with the result that the only remaining open case is for finite von Neumann algebras. Recently Shen \cite{Shen} studied this question and introduced a numerical quantity $\mathcal{G}(M)$ which is related to the number of generators of $M$.  He showed that if $
\mathcal{G}(M)<1/4$, then $M$ is generated by a projection $p$ and a hermitian element $h$ (and thus singly generated by $p+ih$) and noted that the same proof shows single generation when
$\mathcal{G}(M)<1/2$ (see \cite[Theorem 3.1]{DSSW} for a strengthened version). Shen was able to establish $\mathcal{G}(M)=0$ for various II$_1$ factors, hyperfinite, Cartan or nonprime,
thus providing many examples of singly generated factors. Indeed, the question of single generation  is equivalent to whether $\mathcal{G}(M)=0$ for all $M$ \cite{DSSW}, since any
example where $\mathcal{G}(M)>0$ leads to examples that are not even finitely generated. In this section we will examine the behavior of $\mathcal{G}(\cdot)$ under the formation of
crossed products, and this will give new classes of singly generated factors.

For a II$_1$ factor $M$ with normalized trace $\tau$, one way to define the Shen invariant $\mathcal{G}(M)$ is as follows. Consider a fixed element $x\in M$ and a fixed set $P=\sett{p_1,\ldots, p_k}$ of orthogonal
projections summing to $I$ and with equal traces $\tau(p_i)=1/k$. Relative to these projections, $x$ can be written as a $k\times k$ matrix $(x_{ij})$, and $\mathcal{I}(x,P)$ is defined to
be the number of nonzero matrix entries divided by $k^2$. For a finite subset $X=\sett{x_1,\ldots, x_n}\subseteq M$, we define
\begin{equation}\label{eq8.1}
\mathcal{I}(X,P)=\sum_{i=1}^n \mathcal{I}(x_i,P).
\end{equation}
If $M$ is not finitely generated then we set $\mathcal{G}(M)=\infty$; otherwise we define
\begin{equation}\label{eq8.2}
\mathcal{G}(M)=\inf \,\mathcal{I}(X,P)
\end{equation}
taken over all finite sets $X$ of generators and all finite sets of projections $P$ as above. If $P=\sett{p_1,\ldots ,p_k}$ is such a set of projections, then we may form a new set
$Q=\sett{q_1,\ldots ,q_{2k}}$ by splitting each $p_i$ as an orthogonal sum of two projections of trace $1/(2k)$. It is straightforward to verify that $\mathcal{I}(x,Q)\leq
\mathcal{I}(x,P)$, and by applying this halving argument an arbitrary number of times, we see that the sets of projections in \eqref{eq8.2} can be restricted to have sizes that exceed
any given integer.

\begin{lem}\label{lem8.1}
Let $G$ be a countable discrete group that acts on a ${\mathrm{II}}_1$ factor $M$ by outer automorphisms, and let $p\in M$ be a nonzero projection. Then there exists an element $a\in p(\cp)p$ so that every coefficient in the Fourier series of $a=\sum_{g\in G}a_gg$ is nonzero.
\end{lem}

\begin{proof}
Let $\sett{g_1,g_2,\ldots}$ be an enumeration of the group elements. Since $M$ is a factor, there exist elements $x_n\in M$, $n\geq 1$, so that $px_n\alpha_{g_n}(p)\ne 0$, and by scaling we may assume that $\|x_n\|\leq 2^{-n}$. Then $x=\sum_{n\geq 1}x_ng_n$ defines an element of $\cp$, and
\begin{equation}\label{eq8.3}
pxp=\sum_{n=1}^{\infty} px_ng_np=\sum_{n=1}^{\infty} px_n\alpha_{g_n}(p)g_n,
\end{equation}
so all the Fourier coefficients of $pxp$ are nonzero. Then $a=pxp\in p(\cp)p$ is the desired element.
\end{proof}

\begin{thm}\label{thm8.2}
Let $M$ be a ${\mathrm{II}}_1$ factor and let $G$ be a countable discrete group acting on $M$ by outer automorphisms. Then
\begin{equation}\label{eq8.4}
\mathcal{G}(\cp)\leq \mathcal{G}(M).
\end{equation}
\end{thm}

\begin{proof}
The result is clear if $\mathcal{G}(M)=\infty$, so we may assume that this quantity is finite. Given $\vp >0$, choose an integer $K$ so that $K>\sqrt{2/\vp}$. By the remarks preceding
Lemma \ref{lem8.1}, there is a set of projections $P=\sett{p_1,\ldots,p_k}$ with $k\geq K$ and a set of generators $X=\sett{x_1,\ldots,x_n}$ for $M$ so that
\begin{equation}\label{eq8.5}
\mathcal{I}(X,P)<\mathcal{G}(M)+\vp/2.
\end{equation}
By Lemma \ref{lem8.1}, there exists an element $x_{n+1}\in p_1(\cp)p_1$ whose Fourier coefficients are all nonzero.
By Remark \ref{rem4.6}, $Y=\sett{x_1,\ldots,x_{n+1}}$ is a generating set for $\cp$, and
\begin{align}
\mathcal{I}(Y,P)&=\mathcal{I}(X,P)+\mathcal{I}(x_{n+1},P)= \mathcal{I}(X,P)+1/k^2\notag\\
&<\mathcal{G}(M)+\vp/2+1/K^2<\mathcal{G}(M)+\vp.\label{eq8.6}
\end{align}
Thus $\mathcal{G}(\cp)<\mathcal{G}(M)+\vp$, and the result follows since $\vp>0$ was arbitrary.
\end{proof}

\begin{rem}\label{rem8.3}
If $\mathcal{G}(M)=0$, then $\mathcal{G}(\cp)=0$ from \cite{Shen}, so Theorem \ref{thm8.2} gives nothing new in this case. However, there are finitely generated factors $M$ for which
$\mathcal{G}(M)$ is currently unknown. Such examples include free group factors where $\mathcal{G}(L(\mathbb{F}_n)\leq(n-1)/2$ for $n\geq 2$ \cite{DSSW}. Using \cite[Theorem 4.5]{DSSW},
we may tensor with suitably large matrix algebras to obtain factors $M$ whose Shen invariants are yet to be determined but which lie below 1/4. Any crossed product of such an algebra is
singly generated by Theorem \ref{thm8.2}.$\hfill\square$
\end{rem}

\end{document}